\documentclass[12pt]{amsart}

\theoremstyle{definition}
\theoremstyle{plain}

\usepackage[left=2.8cm,top=2.8cm,right=2.8cm,bottom = 2.4cm]{geometry}
\usepackage{amsfonts}
\usepackage{amsmath}
\usepackage{amscd}
\usepackage{amssymb}
\usepackage{amsthm}
\usepackage{bbm}
\usepackage{bm}
\usepackage{enumerate}
\usepackage{enumitem}
\usepackage{hyperref}
\usepackage{latexsym}
\usepackage{mathabx}
\usepackage{mathdots}
\usepackage{mathtools}

\allowdisplaybreaks

\newcommand{\ad}{\textup{ad}}

\newcommand{\bbar}{\overline{b} }
\newcommand{\Bers}{\textup{Bers}}
\newcommand{\C}{\mathbb{C}}
\newcommand{\calC}{\mathcal{C}}
\newcommand{\Can}{\textup{Can}}
\newcommand{\cbar}{\overline{c} }
\newcommand{\Cg}{\mathfrak{C}_{g}}
\newcommand{\Cgn}{\mathfrak{C}_{g,n}}
\newcommand{\Cgone}{\mathfrak{C}_{g,1}}
\newcommand{\Cgtwo}{\mathfrak{C}_{g,2}}
\newcommand{\Chat}{\widehat{\C}}
\newcommand{\D}{\mathcal{D}}
\newcommand{\del}{\partial}
\newcommand{\delCg}{\nabla_{\hspace{-1 mm}\Cg}}
\newcommand{\delCgcheck}{\widecheck{\nabla}_{\hspace{-1 mm}\Cg}}
\newcommand{\delCgn}{\nabla_{\hspace{-1 mm}\Cgn}}
\newcommand{\delCgncheck}{\widecheck{\nabla}_{\hspace{-1 mm}\Cgn}}

\newcommand{\delCgone}{\nabla_{\hspace{-1 mm}\Cgone}}
\newcommand{\delCgtwo}{\nabla_{\hspace{-1 mm}\Cgtwo}}

\newcommand{\delMg}{\nabla_{\hspace{-1 mm}\Mg}}
\newcommand{\delMgn}{\nabla_{\hspace{-1 mm}\Mgn}}

\newcommand{\E}{\mathcal{E}}
\newcommand{\End}{\textup{End}}
\newcommand{\F}{\mathcal{F}}
\newcommand{\I}{\mathcal{I}}
\newcommand{\Id}{\textup{Id}}
\newcommand{\im}{\textup{i}}
\newcommand{\Ip}{\I_{+}}
\newcommand{\J}{\mathcal{J}}
\newcommand{\Mg}{{\mathcal M}_{g}}
\newcommand{\Mgn}{{\mathcal M}_{g,n}}

\newcommand{\Omo}{\Omega_{0}}
\newcommand{\Sg}{\mathcal{S}_{g}}
\newcommand{\Sgn}{\mathcal{S}_{g,n}}
\newcommand{\Szero}{\mathcal{S}_{0}}
\newcommand{\SL}{\textup{SL}}
\newcommand{\slLie}{\mathfrak{sl}}

\newcommand{\Sym}{\textup{Sym}}
\newcommand{\tpi}{2\pi \im} 
\newcommand{\Tr}{\textup{tr}}
\newcommand{\vac}{\mathbbm{1}}
\newcommand{\wt}{\textup{wt}}
\newcommand{\Z}{\mathbb{Z}}
\newcommand{\Zhu}{\textup{Zhu}}

\newtheorem{corollary}{Corollary}[section]
\newtheorem{example}{Example}[section]
\newtheorem{lemma}{Lemma}[section]

\newtheorem{proposition}{Proposition}[section]
\newtheorem{remark}{Remark}[section]
\newtheorem{theorem}{Theorem}[section]

\title{General Genus Zhu Recursion for Vertex Operator Algebras\\
}
\date{}
\begin{document}
	\author{Michael P. Tuite}
	\address{School of Mathematics, Statistics and Applied Mathematics\\ 
		National University of Ireland Galway, Galway, Ireland}	
	\email{michael.tuite@nuigalway.ie}
	\author[Michael Welby]{Michael Welby$^{\dagger}$}
	\thanks{$^{\dagger}$Supported by an Irish Research Council Government of Ireland Postgraduate Scholarship.}
	\address{School of Mathematics, Statistics and Applied Mathematics\\ 
		National University of Ireland Galway, Galway, Ireland}
\email{m.welby5@nuigalway.ie}

\begin{abstract}
	We describe Zhu recursion  for a vertex operator algebra
	(VOA) on a general genus Riemann surface in the Schottky uniformization
	where  $n$-point correlation functions are written as linear combinations of 
	$(n-1)$-point functions with universal coefficients. These coefficients are identified with  specific geometric structures on the Riemann surface.
	We  apply Zhu recursion to the Heisenberg VOA and determine all its correlation functions. 
	For a general VOA, Zhu recursion with respect to the Virasoro vector is shown to lead to conformal Ward identities expressed in terms of derivatives with respect to the surface moduli. 
	We derive linear partial differential equations for the Heisenberg VOA partition function and various 
	structures such as the  bidifferential of the second kind, holomorphic 1-forms and the period matrix. 
	We also compute the genus $g$ partition function for an even lattice VOA.
\end{abstract}
\maketitle
\section{Introduction}
Conformal field theories on Riemann surfaces have been studied by physicists for over 30 years e.g. \cite{AGMV, EO}. The purpose of this paper is to make some of these ideas explicitly realizable within Vertex Operator Algebra  (VOA) theory  e.g.~\cite{K,LL,MT1}. This work follows earlier results for genus two \cite{MT2,MT3,GT} and higher \cite{TZ,T2}.
In his seminal paper~\cite{Z}, Zhu describes a recursion formula for genus one $n$-point correlation functions for a VOA  in terms of a finite linear combination of $(n-1)$-point functions with universal coefficients given by Weierstrass elliptic functions.
Zhu recursion provides a powerful calculational tool which, for example, was used by Zhu to prove  convergence and modularity of the simple module partition functions for rational $C_{2}$ cofinite 
VOAs~\cite{Z}. Zhu recursion was recently extended to correlation functions on a genus two surface formed by sewing two tori~\cite{GT}. In this work we describe Zhu recursion for VOA correlation functions on a general genus  Riemann surface  in the Schottky uniformization. 
We obtain higher genus universal coefficients analogous to Weierstrass functions which are described by explicit geometric objects e.g. we prove that 1-point functions are determined by holomorphic differentials. We consider applications to the Heisenberg VOA and the development of conformal Ward identities expressed in terms of explicit  differential operators giving the variation with respect to the Riemann surface moduli.   
 

In Section 2  we briefly review  the classical Schottky  uniformization of  a genus $g$ Riemann surface  $\mathcal{S}_{g}$~\cite{Fo,FK, Bo}. In particular, we use non-standard Schottky parameters suitable for our later VOA calculations.

Section 3 begins with some standard VOA theory e.g.~\cite{K,LL,MT1}. 
We develop genus zero Zhu recursion theory for an $n$-point function involving a conformal weight $N$ quasiprimary vector which is expressed as a sum of $(n-1)$-point functions with VOA independent coefficients which are not unique due to an $(2N-1)$-fold correlation function symmetry.

Section 4 is concerned with Zhu recursion on $\mathcal{S}_{g}$. We define the genus $g$ partition and $n$-point functions for simple self-dual VOAs  of strong CFT-type (with a unique invertible invariant bilinear form~\cite{FHL,Li}) in terms of infinite sums of certain genus zero correlation functions (that depend on Schottky parameters and Riemann surface points). We prove that these definitions agree  with Zhu's trace functions at genus one~\cite{Z}. We describe the M\"obius $\SL_{2}(\C)$ covariance properties of genus $g$ correlation functions.  
Our main result (Theorem~\ref{theor:ZhuGenusg})  describes the genus $g$ Zhu recursion formula. The proof relies partly on  genus zero Zhu recursion but  also on other recursion properties following from the invariance of the   bilinear form. We find that the coefficients in our Zhu recursion formula contain universal terms, $\Psi_{N}(x,y)$ and $\Theta_{a}(x,\ell)$, expressed as certain formal series. $\Psi_{N}(x,y)$ is the genus $g$ analogue of the leading Weierstrass term in the original genus one Zhu  recursion formula.  

In Section~5  we identify the universal coefficients $\Psi_{N}(x,y)$ and $\Theta_{a}(x,\ell)$ for genus $g\ge 2$.  In particular,  $\Psi_{N}(x,y)$ is shown to be a convergent Poincar\'{e} sum over the genus $g$ Schottky group with summand chosen by exploiting the $(2N-1)$-fold symmetry for genus zero correlation functions. 
For $N=1$ we find that $\Psi_{1}(x,y)$ is the classical differential of the third kind~\cite{Fa,Bo} with a simple pole  at $x=y$. 
For $N\ge 2$ we find $\Psi_{N}(x,y)$ is meromorphic with a simple pole at $x=y$. It is an $N$-differential in $x$ and a $(1-N)$-quasidifferential in $y$ with cocycles as introduced by Bers \cite{Be1} and is referred to as a GEM form (Green's function with Extended Meromorphicity)  in~\cite{T1}. 
The other universal coefficients $\Theta_{a}(x,\ell)$ (where $a=1,\ldots,g$ and $\ell=1,\ldots, 2N-2$) are related to $\Psi_{N}(x,y)$  and define a spanning set of holomorphic $N$-differentials on  $\mathcal{S}_{g}$. In particular, these describe the space of 1-point functions for quasiprimary vectors of weight $N$. 
An alternative canonical choice for the GEM form $\Psi_{N}^{\Can}(x,y)$ is also discussed where the coefficients $\Theta_{a}(x,\ell)$ are replaced by a $(g-1)(2N-1)$-dimensional canonically normalized basis of holomorphic $N$-differentials. 

Section 6 discusses applications of genus $g$ Zhu recursion. 
We determine all $n$-point functions for the Heisenberg VOA in terms of the classical bidifferential of the second kind extending previous results at genus two~\cite{MT2, MT3,GT}. 
Zhu recursion with respect to the Virasoro vector for a general VOA leads to various conformal Ward identities. 
We find that the Virasoro $1$-point function is given by the action of a canonical differential operator on the partition function describing its variation with respect to the surface moduli with coefficients given by holomorphic 2-differentials generalizing genus two results of~\cite{GT}. 
The deeper relationship of this differential operator to Beltrami differentials \cite{A} is discussed  in \cite{T1}.  
We derive  linear partial differential equations for the Heisenberg  VOA partition function and some standard objects~\cite{Fa,Mu,Bo} on $\mathcal{S}_{g}$ such as the  bidifferential of the second kind, the projective connection, holomorphic 1-differentials and the period matrix - the latter case being a reformulation of Rauch's identity~\cite{R,O}. 
We conclude with an expression for the partition function for an even lattice VOA in terms of the Siegel lattice theta function and  the Heisenberg partition function, which generalises results of~\cite{MT2, MT3}. 

\noindent \textbf{Acknowledgements.} 
	We thank Tom Gilroy, Geoff Mason and Hiroshi Yamauchi for a number of very helpful comments and suggestions.

\section{The Schottky Uniformization of Riemann Surfaces}
\label{sec:Genus g}
\subsection{The Schottky uniformization of a Riemann surface}
\label{sec:Schottky}
Consider a compact marked  Riemann surface $\Sg$ of genus $g$, e.g.~\cite{FK,Mu,Fa,Bo},  
with canonical homology basis $\alpha_{a}$, $\beta_{a}$ for $a\in\Ip=\{1,2,\ldots,g\}$. 
We review the  construction of a genus $g$ Riemann surface $\Sg$ using the Schottky uniformization where we sew $g$ handles to the Riemann sphere $\Szero\cong\Chat:=\C\cup \{\infty\}$ e.g.~\cite{Fo, Bo}. Every Riemann surface can be (non-uniquely) Schottky uniformized~\cite{Be2}.

For  $a\in\I=\{\pm 1,\pm 2,\ldots,\pm g\}$  let $\calC_{a}\subset \Szero$ be $2g$ non-intersecting Jordan curves where $z\in \calC_{a}$, $z'\in \calC_{-a}$ for $a\in\Ip$ are identified by a sewing relation
\begin{align}\label{eq:SchottkySewing}
\frac{z'-W_{-a}}{z'-W_{a}}\cdot\frac{z-W_{a}}{z-W_{-a}}=q_{a},\quad a\in\Ip,
\end{align}
for $q_{a}$ with $0<|q_{a}|<1$ and $W_{\pm a}\in\Chat$. 
Thus  $z'=\gamma_{a}z$  for $a\in\Ip$ with 
\begin{align}\label{eq:gammaa}
\gamma_{a}:=\sigma_{a}^{-1}
\begin{pmatrix}
q_{a}^{1/2} &0\\
0 &q_{a}^{-1/2}
\end{pmatrix}
\sigma_{a},\quad 
\sigma_{a}:=(W_{-a}-W_{a})^{-1/2}\begin{pmatrix}
1 & -W_{-a}\\
1 & -W_{a}
\end{pmatrix}.
\end{align}
$\sigma_{a}(W_{-a})=0$ and $\sigma_{a}(W_{a})=\infty$ are, respectively, attractive and repelling fixed points of $Z\rightarrow Z'=q_{a}Z$ for  $Z=\sigma_{a} z$ and $Z'=\sigma_{a} z'$.
$W_{-a}$ and  $W_{a}$  are the corresponding fixed points for $\gamma_{a}$. We identify the standard homology cycles $\alpha_{a}$  with $\calC_{-a}$ and  $\beta_{a}$ with a path connecting  $z\in  \calC_{a}$ to $z'=\gamma_{a}z\in  \calC_{-a}$.

The genus $g$ Schottky group $\Gamma$  is   the free group with generators $\gamma_{a}$.
Define $\gamma_{-a}:=\gamma_{a}^{-1}$. The independent elements of $\Gamma$ are reduced words of length $k$
 of the form $\gamma=\gamma_{a_{1}}\ldots \gamma_{a_{k}}$ where $a_{i}\neq -a_{i+1}$ for each $i=1,\ldots ,k-1$. 
 
 Let $\Lambda(\Gamma)$ denote the  limit set  of $\Gamma$ i.e. the set of limit points of the action of $\Gamma$ on $\Chat$. Then $\Sg\simeq\Omo/\Gamma$ where $\Omo:=\Chat-\Lambda(\Gamma)$.
  We let $\D\subset\Chat$  denote the standard connected fundamental region with oriented boundary curves $\calC_{a}$. 

Define $w_{a}:=\gamma_{-a}.\infty$. Using~\eqref{eq:SchottkySewing} we find 
\begin{align}
\label{eq:wa}
w_{a}=\frac{W_{a}-q_{a}W_{-a}}{1-q_{a}},\quad a\in\I ,
\end{align}
where we define $q_{-a}:=q_{a}$.
Then \eqref{eq:SchottkySewing} is equivalent to
\begin{align}\label{eq:SchottkySewing2}
(z'-w_{-a})(z-w_{a})=\rho_{a},
\end{align}
with 
\begin{align}
\label{eq:rhoa}
\rho_{\pm a}:=-\frac{q_{a}(W_{a}-W_{-a})^{2}}{(1-q_{a})^{2}}.
\end{align}
\eqref{eq:SchottkySewing2} implies
\begin{align}
\label{eq:gamma_a.z}
\gamma_{a}z=w_{-a}+\frac{\rho_{a}}{z-w_{a}}.
\end{align}
 Let $\Delta_{a}$ be the disc with centre $w_{a}$ and radius $|\rho_{a}|^{\frac{1}{2}}$. 
It is convenient to choose the Jordan curve $\calC_{a}$ to be the  boundary of $\Delta_{a}$. Then 
$\gamma_{a}$ maps the exterior (interior) of $\Delta_{a}$  to the interior (exterior) of $\Delta_{-a}$ since
 \begin{align*}
 |\gamma_{a}z-w_{-a}||z-w_{a}|=|\rho_{a}|.
 \end{align*}
 Furthermore,  the discs $\Delta_{a},\Delta_{b}$ are non-intersecting if and only if
 \begin{align}
 \label{eq:JordanIneq}
 |w_{a}-w_{b}|>|\rho_{a}|^{\frac{1}{2}}+|\rho_{b}|^{\frac{1}{2}},\quad \forall \;a\neq b.
 \end{align} 
We define $\mathfrak{C}_{g}$ to be the set 
$\{ (w_{a},w_{-a},\rho_{a})|a\in\Ip\}\subset \C^{3g}$ satisfying \eqref{eq:JordanIneq}. We refer to $\mathfrak{C}_{g}$ as the Schottky parameter space.

The relation \eqref{eq:SchottkySewing} is M\"obius  invariant for $\gamma =\left(\begin{smallmatrix}A&B\\C&D\end{smallmatrix}\right)\in\SL_{2}(\C)$ with $(z,z',W_{a},q_{a})\rightarrow(\gamma z,\gamma z',\gamma W_{a},q_{a})$ giving
an $\SL_{2}(\C)$ action on $\mathfrak{C}_{g}$ as follows
\begin{align}
\label{eq:Mobwrhoa}
\gamma:(w_{a},\rho_{a})\mapsto & 
\left(	\frac { \left( Aw_{a}+B \right)  \left( Cw_{-a}+D \right) -\rho_{a}
	\,AC}{ \left( Cw_{a}+D \right)  \left( Cw_{-a}+D \right) -\rho_{a}\,{
		C}^{2}},
{\frac {\rho_{a}}{ \left(  \left( Cw_{a}+D \right)  \left( Cw_{-a}+D
		\right) -\rho_{a}\,{C}^{2} \right) ^{2}}}\right).
\end{align}
We define Schottky space as $\mathfrak{S}_{g}:=\mathfrak{C}_{g}/\SL_{2}(\C)$ which provides a natural covering space for the moduli space of genus $g$ Riemann surfaces (of dimension 1 for $g=1$ and $3g-3$ for $g\ge 2$).  Despite the more complicated M\"obius action,  
we exploit the $\mathfrak{C}_{g}$  parameterization in this work because the sewing relation  \eqref{eq:SchottkySewing2} is more  readily implemented in the theory of vertex operators than the traditional relation  \eqref{eq:SchottkySewing}. 

\section{Vertex Operator Algebras and Genus Zero Zhu Recursion}
\label{sec:VOAZhu0}
\subsection{Vertex Operator Algebras} 
We  review aspects of vertex operator algebras e.g.~\cite{K,FHL,LL,MT1}. A Vertex Operator Algebra (VOA) is a quadruple $(V,Y(\cdot,\cdot),\vac,\omega)$ consisting of a graded vector space $V=\bigoplus_{n\ge 0}V_{n}$  with two distinguished elements: the vacuum vector $\vac\in V_{0}$  and the Virasoro conformal vector $\omega\in V_{2}$.  
For each $u\in V$ there exists a vertex operator, a  Laurent series in a formal variable $z$ given by   
\begin{align*}
Y(u,z)=\sum_{n\in\Z}u(n)z^{-n-1},
\end{align*}
for \emph{modes} $u(n)\in\End(V)$ 
where  $u=u(-1)\vac $ and $u(n)\vac =0\;\forall\; n\ge 0$.
Vertex operators are local: for each  $u,v\in V$  there exists an integer $N\ge 0$ such that
\begin{align*}
(x-y)^{N}[Y(u,x),Y(v,y)]=0.
\end{align*}
For the Virasoro conformal vector  $\omega$
\begin{align*}
Y(\omega,z)&=\sum_{n\in\Z}L(n)z^{-n-2},
\end{align*}
where the operators  $L(n)=\omega(n+1)$  satisfy the Virasoro  algebra
\begin{align*}
[L(m),L(n)]=(m-n)L(m+n)+\frac{m^{3}-m}{12}\delta_{m,-n}C,
\end{align*}
for a constant \emph{central charge} $C$.   Vertex operators  satisfy translation
\begin{align*}
Y(L(-1)u,z)=\del_{z}Y(u,z).
\end{align*}
$V_{n}=\{v\in V:L(0)v=nv\}$ with $\dim V_{n}<\infty$. 
$v\in V_{n}$ is said to have \emph{(conformal) weight} $\wt(v)=n$. 

\medskip
We quote a number of basic VOA properties e.g.~\cite{K,FHL,LL,MT1}. For $v\in V_{k}$ 
\begin{align}
\label{eq:vnVm}
v(n):V_{m}\rightarrow V_{m+k-n-1}.
\end{align}
The commutator identity: for all $u,v\in V$ we have 
\begin{align}\label{CommutatorId1}
[u(m),Y(v,z)]&=\sum_{j\ge 0}\binom{m}{j}Y(u(j)v,z)z^{m-j} 
\\
\label{CommutatorId2}
&=\left (\sum_{j\ge 0}Y(u(j)v,z)\partial_{z} ^{(j)}\right)z^{m},
\end{align} 
where $\del_{z}^{(j)}:=\tfrac{1}{j!}\del_{z}^{j}$. The expression \eqref{CommutatorId2} is particularly useful for our later purposes. 
\\
The associativity identity: for each $u,v\in V$  there exists $M\ge 0$ such that
\begin{align}
(x+y)^{M}Y(Y(u,x)v,y)&=(x+y)^{M}Y(u,x+y)Y(v,y).
\label{AssocId}
\end{align}

There is a M\"obius  action for $\gamma=\left(\begin{smallmatrix}a&b\\c&d\end{smallmatrix}\right)\in\SL_{2}(\C)$  on vertex operators ~\cite{DGM,K} 
\begin{align}\label{eq:MobV}
\gamma:Y(u,z)\rightarrow D_{\gamma}  Y(u,z) D_{\gamma}^{-1} =Y\left(e^{-c(cz+d)L(1)}\left(cz+d\right)^{-2L(0)}u,\frac{az+b}{cz+d}\right),
\end{align}
where $
D_{\gamma}=e^{\frac{b}{d}L(-1)}d^{-2L(0)} e^{-\frac{c}{d}L(1)}
$
with $D_{\gamma}\vac=0$. 
In particular, associated with the M\"obius map $z\rightarrow \frac{\rho}{z}$, for  a given scalar $\rho$,
we define an adjoint vertex operator ~\cite{FHL, Li}
\begin{align}\label{eq:Yadj}
Y_{\rho}^{\dagger}(u,z):=\sum_{n\in\Z}u_{\rho}^{\dagger}(n)z^{-n-1}=Y\left(e^{\frac{z}{\rho}L(1)}\left(-\frac{\rho}{z^{2}}\right)^{L(0)}u,\frac{\rho}{z}\right).
\end{align}
For quasiprimary $u$ (i.e. $L(1)u=0$) of weight $N$ we have
\begin{align}
\label{eq:udagger}
u_{\rho}^{\dagger}(n)=(-1)^{N}\rho^{n+1-N}u(2N-2-n),
\end{align}
e.g. $L_{\rho}^{\dagger}(n)=\rho^{n}L(-n)$.
A bilinear form $\langle \cdot,\cdot\rangle_{\rho}$ on $V$  is said to be invariant if
\begin{align}
\label{eq:LiZinvar}
\langle Y (u, z)v, w\rangle_{\rho} = \langle v, Y_{\rho}^{\dagger}
(u, z)w\rangle_{\rho},\quad \forall\;u,v,w\in V.
\end{align}
$\langle \cdot,\cdot\rangle_{\rho}$ is symmetric where $\langle u,v\rangle_{\rho}=0$ for $\wt(u)\neq\wt(v)$~\cite{FHL} and 
\begin{align}
\label{eq:uvA}
\langle u,v\rangle_{\rho}=\rho^{N}\langle u,v\rangle_{1},
\quad N=\wt(u)=\wt(v).
\end{align} 
We assume throughout that $V$ is of strong CFT-type i.e. $V_{0} = \C\vac$ and $L(1)V_{1}= 0$. Then the bilinear form with normalisation $\langle \vac,\vac\rangle_{\rho}=1$ is unique~\cite{Li}. We also assume that $V$ is simple and  self-dual  ($V$ is isomorphic to the contragredient module as a $V$-module).
Then the bilinear form is non-degenerate~\cite{Li}. We refer to this unique invariant non-degenerate bilinear form as the Li-Zamolodchikov (Li-Z) metric.

\subsection{Genus Zero Zhu Recursion}
We develop genus zero Zhu theory in more generality than in ref.~\cite{Z}. In particular, we describe a $(2N-1)$-fold symmetry for genus zero $n$ point functions associated with a quasiprimary vector of weight $N$. This is exploited in the later genus $g$ analysis.

Define the genus zero $n$-point (correlation) function for $v_{1},v_{2},\ldots, v_{n}\in V$ for formal $y_{1},\ldots,y_{n}$ by
\begin{align*}
Z^{(0)}(\bm{v,y}):=Z^{(0)}(\ldots;v_{k},y_{k};\ldots)=\langle \vac,\bm{Y(v,y)}\vac\rangle,
\end{align*}
where  $\langle\cdot,\cdot\rangle$ is the Li-Z metric with $\rho=1$ and 
\begin{align*}
\bm{Y(v,y)}:= Y(v_{1},y_{1}) \ldots Y (v_{n},y_{n}).
\end{align*}
$Z^{(0)}(\bm{v,y})$ depends only on the $n$ given vertex operators and is a rational function of $y_{1},\ldots,y_{n}$ which can thus  be interpreted as points  on $\Chat$ (cf. Remark~\ref{rem:flzero}).
We also define the $n$-point correlation differential form
\begin{align}\label{eq:BFForm}
\F^{(0)}(\bm{v,y}):=Z^{(0)}(\bm{v,y})\bm{dy^{\wt(v)}},
\end{align}
where $\bm{dy^{\wt(v)}}:=\prod_{k=1}^{n}dy_{k}^{\wt(v_{k})}$  for homogeneous $v_{k}$  extended by linearity for general $v_{k}$. 

Let  $u$ be quasiprimary of weight $N$. Using \eqref{eq:udagger} we find
\begin{align*}
u(\ell)\vac=
u^{\dagger}(\ell)\vac=0,
\end{align*}
for all $0\le \ell \le 2N-2$ so that
\begin{align*}
0=\langle u^{\dagger}(\ell)\vac,\bm{Y(v,y)}\vac\rangle=
\langle \vac,u(\ell)\bm{Y(v,y)}\vac\rangle=
\langle \vac,\left[u(\ell),\bm{Y(v,y)}\right]\vac\rangle.
\end{align*}
Apply the commutator identity~\eqref{CommutatorId2} this  implies that for $0\le \ell \le 2N-2$
\begin{align}
\label{ZhuPrepLemma}
\sum_{k=1}^n \left(\sum_{j\ge 0}
Z^{(0)}(\ldots;u(j)v_{k},y_{k};\ldots)\partial_{k} ^{(j)} \right) y_{k}^{\ell}=0 ,
\end{align}
where $\del_{k}^{(j)}:=\tfrac{1}{j!}\del_{y_{k}}^{j}$. 
Thus we find a genus zero analogue of~\cite{Z}, Proposition 4.3.1.
\begin{lemma}\label{lem:ZeroSum}
	Let  $u$ be quasiprimary of weight $N$. Then 
\begin{align*}
\sum_{k=1}^n \left(\sum_{j\ge 0}
Z^{(0)}(\ldots;u(j)v_{k},y_{k};\ldots)\partial_{k} ^{(j)} \right)p(y_{k})
= 0,
\end{align*}
for any $p(y)\in\Pi_{2N-2}(y)$, the space of degree $2N-2$ polynomials in $y$.
\end{lemma}
\begin{remark}
For the conformal vector $u=\omega$ with $N=2$ then Lemma~\ref{lem:ZeroSum} is the genus zero conformal Ward identity.
\end{remark}	
We now develop a genus zero Zhu recursion formula for $Z^{(0)}(u,x;\bm{v,y})$ for  quasiprimary $u$ of weight $N$. First note that for all $v\in V$ and $s<0$
\begin{align*}
0=\langle u(-s-1)\vac,v\rangle=
\langle \vac,u^{\dagger}(-s-1)v\rangle=(-1)^{N}\langle\vac,u(s+2N-1)v\rangle.
\end{align*}
Hence we find using commutativity \eqref{CommutatorId2} that 
\begin{align*}
Z^{(0)}(u,x;\bm{v,y})
&=\sum_{s\ge 0}x^{-s-2N}\left\langle\vac,u(s+2N-1)\bm{Y(v,y)}\vac\right\rangle
\\
&=\sum_{k=1}^{n}\sum_{s\ge 0}x^{-s-2N}
\left\langle\vac,
\ldots\sum_{j\ge 0}\del_{k}^{(j)}\left(y_{k}^{s+2N-1}\right)Y(u(j)v_{k},y_{k})\ldots \vac\right\rangle
\\
&=\sum_{k=1}^{n}\sum_{j\ge 0}x^{-1}
\del_{k}^{(j)}\left(\sum_{s\ge 0}\left(\frac{y_{k}}{x}\right)^{s+2N-1}
\right)
Z^{(0)}(\ldots;u(j)v_{k},y_{k};\ldots)
\\
&=\sum_{k=1}^{n}\sum_{j\ge 0}\del^{(0,j)}\zeta_{N}(x,y_{k})Z^{(0)}(\ldots;u(j)v_{k},y_{k};\ldots),
\end{align*}
where
\begin{align*}
\zeta_{N}(x,y):=x^{-1}\sum_{s\ge 0}\left(\frac{y}{x}\right)^{s+2N-1}
=\left(\frac{y}{x}\right)^{2N-1}\cdot\frac{1}{x-y} =\frac{1}{x-y}-\sum_{\ell=0}^{2N-2}\frac{y^{\ell}}{x^{\ell+1}},
\end{align*}
and where $\del^{(i,j)}f(x,y):=\del_{x}^{(i)}\del_{y}^{(j)}f(x,y)$ for a function $f(x,y)$. Here and below we use the formal binomial expansion convention that 
\begin{align}
\label{eq:binexp}
(x+y)^{m}=\sum_{k\ge 0}\binom{m}{k}x^{m-k}y^{k},
\end{align} 
for $m\in\Z$ and formal $x,y$.  
Lemma~\ref{lem:ZeroSum} implies that we may replace $\zeta_{N}(x,y)$ by
\begin{align}\label{PsiDef}
\psi_{N}^{(0)}(x,y):=\frac{1}{x-y}+\sum_{\ell=0}^{2N-2}f_{\ell}(x)y^{\ell},
\end{align}
for \emph{any} Laurent series $f_{\ell}(x)$ for $\ell=0,\ldots ,2N-2$. Thus we find
\begin{theorem}[Quasiprimary Genus Zero Zhu Recursion]
	\label{theor:GenusZeroZhu}
Let $u$ be quasiprimary of weight $N$. Then
\begin{align}\label{GenusZeroZhu}
Z^{(0)}(u,x;\bm{v,y})=\sum_{k=1}^{n}\sum_{j\ge 0}\del^{(0,j)}\psi_{N}^{(0)}(x,y_{k})Z^{(0)}(\ldots;u(j)v_{k},y_{k};\ldots).
\end{align}
\end{theorem}
Since $Y\left(\frac{1}{i!}L(-1)^{i}u,x\right)=\del_{x}^{(i)}Y(u,x)$ we  find in general that
\begin{corollary}[General Genus Zero Zhu Recursion]
	\label{cor:GenGenusZeroZhu}
	 Let $\frac{1}{i!}L(-1)^{i}u$ be a  quasiprimary descendant of $u$ with $\wt(u)=N$. Then
\begin{align*}
&\tfrac{1}{i!}Z^{(0)} \left(L(-1)^{i}u,x;\bm{v,y}\right)=\sum_{k=1}^{n}\sum_{j\ge 0}\partial^{(i,j)}\psi_{N}^{(0)}(x,y_{k})Z^{(0)} (\ldots;u(j)v_{k};y_{k};\ldots).
\end{align*}
\end{corollary}
\begin{remark}\label{rem:flzero}
Choosing $f_{\ell}(x)=0$ we obtain the neatest  form of Zhu recursion
(which implies that genus zero $n$-point functions are rational) using
\[
\del^{(i,j)}\psi_{N}^{(0)}(x,y)=(-1)^{i}\binom{i+j}{i}\frac{1}{(x-y)^{1+i+j}}.
\] 
However for our later purposes it is more useful to consider more general   $\psi_{N}^{(0)}(x,y)$.  
\end{remark}
Associativity \eqref{AssocId} implies (up to a rational multiplier) that
\begin{align}
Z^{(0)}(u,x+y_{1};\bm{v,y})
&=Z^{(0)}(Y(u,x)v_{1},y_{1};v_{2},y_{2}\ldots ;v_{n},y_{n})
\notag
\\
&=\sum_{m\in\Z}Z^{(0)}(u(m)v_{1},y_{1};\ldots )x^{-m-1}.
\label{eq:Z0assoc}
\end{align}
But Zhu recursion \eqref{GenusZeroZhu}  implies that 
\begin{align*}
Z^{(0)}(u,x+y_{1};\bm{v,y})
&=
\sum_{j\ge 0}\del^{(0,j)}\psi_{N}^{(0)}(y_{1}+x,y_{1})Z^{(0)}(u(j)v_{1},y_{1};\ldots)
\\
&\quad+\sum_{k=2}^{n}\sum_{j\ge 0}\del^{(0,j)}\psi_{N}^{(0)}(y_{1}+x,y_{k})Z^{(0)}(\ldots;u(j)v_{k},y_{k};\ldots).
\end{align*}
Considering the $x$ expansions 
\begin{align*}
\del^{(0,j)}\psi_{N}^{(0)}(y+x,z)&=\sum_{t\ge 0}\del^{(t,j)}\psi_{N}^{(0)}(y,z)x^{t},\quad y\neq z,
\\ 
\partial^{(0,j)}\psi_{N}^{(0)}(y+x,y)&=\frac{1}{x^{1+j}}+\sum_{\ell=0}^{2N-2}\binom{\ell}{j}f_{\ell}(y+x)y^{\ell-j}=\frac{1}{x^{1+j}}+\sum_{t\ge 0}\E^{j}_{t}(y)x^{t},
\end{align*}
where
\begin{align}
\label{eq:Ejt}
\E_{t}^{j}(y)=\sum_{\ell=0}^{2N-2}\del^{(t)}f_{\ell}(y)\;\del^{(j)}y^{\ell},
\end{align}
and comparing the coefficients of $x^{t}$ for $t\ge 0$ we obtain
\begin{theorem}[Quasiprimary Genus Zero Zhu Recursion II]\label{GenusZeroZhu3}
Let $u$ be quasiprimary of weight $N$. Then for all $t\ge 0$ we have
\begin{align*}
Z^{(0)}(u(-t-1)v_{1},y_{1};\ldots ) 
=&\sum_{j\ge 0}\E_{t}^{j}(y_{1})Z^{(0)}(u(j)v_{1},y_{1};\ldots )
\\
&+\sum_{k=2}^{n}\sum_{j\ge 0}\del^{(t,j)}\psi_{N}^{(0)}(y_{1},y_{k})Z^{(0)}(\ldots ;u(j)v_{k},y_{k};\ldots) .
\end{align*}
\end{theorem}
\begin{remark}
	\label{rem:Zhu genus zero II}
Using locality, we can allow the $u(-t-1)$  mode to act on any $v_{k}$ for $k = 1, \ldots, n$ and adjust the right hand side accordingly. A similar recursion formula for a general quasiprimary descendant can be found.
\end{remark}
These results can be expressed in terms of formal differential forms \eqref{eq:BFForm}. Thus with  $\Psi_{N}^{(0)}(x,y)=\psi_{N}^{(0)}(x,y)dx^{N}dy^{1-N}$  then Theorem~\ref{theor:GenusZeroZhu} is equivalent to
\begin{corollary}
	For $u$ quasiprimary of weight $N$ we have
	\begin{align}\label{GenusZeroZhuForms}
	\F^{(0)}(u,x;\bm{v,y})=\sum_{k=1}^{n}\sum_{j\ge 0}\del^{(0,j)}\Psi_{N}^{(0)}(x,y_{k})dy_{k}^{j}\,\F^{(0)}(\ldots;u(j)v_{k},y_{k};\ldots).
	\end{align}
\end{corollary}

\section{Genus $g$ Zhu Recursion}
\label{sec:gZhu}
In this section  we introduce formal partition and $n$-point correlation functions for a VOA associated to a genus $g$ Riemann surface $\Sg$ in the Schottky  scheme  with sewing relation \eqref{eq:SchottkySewing2} of \S\ref{sec:Genus g}. The approach taken is a generalization of the genus two sewing schemes of~\cite{MT2,MT3} and genus two Zhu recursion of~\cite{GT}.
Initially, all expressions will be functions of formal variables $w_{\pm a},\rho_{a}$ and vertex operator parameters. We later describe a genus $g$ Zhu recursion formula with universal  coefficients that have a geometrical meaning and are meromorphic on $\Sg$ for all $(w_{\pm a},\rho_{a})\in\mathfrak{C}_{g}$. This is a generalization of the genus zero situation above and for genus one where one finds elliptic Weierstrass functions~\cite{Z}. 
From this section onwards (except for $n$-point functions) genus $g$ objects will normally carry no genus label whereas genus zero objects will be notated e.g. $\psi_{N}^{(0)}$.  
\subsection{Genus $g$ formal partition and $n$-point functions}
We first note an important lemma which we exploit below. Recall that  the Li-Z metric $\langle\cdot, \cdot\rangle$ of \eqref{eq:LiZinvar} (where we suppress the $\rho$ subscript for clarity) is invertible and that $\langle u,v \rangle=0$ for  $\wt(u)\neq \wt(v)$ for homogeneous $u,v$.  
Let $\{ b\}$ be a homogeneous basis for  $V$ with Li-Z dual basis $\{\bbar \}$. 
\begin{lemma}\label{AdjointLemma}
	For $u$ quasiprimary of weight $N$ we have
	\begin{align}
	\label{eq:adjointrel}
	\sum_{b\in V_n} \left(u(m)b\right)\otimes \bbar =\sum_{b\in V_{n+N-m-1}} b\otimes \left(u^\dagger(m)\bbar \right).
	\end{align}
\end{lemma}
\begin{proof} Since $\wt(u(m)b)=n+N-m+1$ for $b\in V_n$ we find
	\begin{align*}
	u(m)b =\sum_{c\in V_{n+N-m-1}}\langle \cbar , u(m)b\rangle c 
	= \sum_{c\in V_{n+N-m-1}}\langle u^\dagger(m)\cbar , b\rangle c.
	\end{align*}
	Similarly for $c\in V_{n+N-m-1}$ we find $\wt(u^\dagger(m)\cbar )=n$  from \eqref{eq:udagger} so that  
	\[
	u^\dagger(m)\cbar =\sum_{b\in V_n}\langle u^\dagger(m)\cbar , b\rangle \bbar .
	\]
	Hence the result follows on relabelling.
\end{proof}
\begin{remark}
	\label{rem:adjoint}
Suppose that $U$ is a subVOA of $V$ and  $\mathcal{W}\subset V$ is a $U$-module. For $u\in U$ and homogeneous $\mathcal{W}$-basis $\{ w\}$  we may then extend  \eqref{eq:adjointrel} to obtain  
	\begin{align*}
\sum_{w\in \mathcal{W}_{n}} \left(u(m)w\right)\otimes \overline{w}=\sum_{w\in \mathcal{W}_{n+N-m-1}} w\otimes \left(u^\dagger(m)\overline{w}\right).
\end{align*} 
\end{remark}
For each $a\in\Ip$, let $\{b_{a}\}$  denote a homogeneous  $V$-basis and let $\{\bbar _{a}\}$ be the  dual basis with respect to the Li-Z metric $\langle \cdot,\cdot\rangle_{1}$  i.e. with $\rho=1$.  Define
\begin{align}
\label{eq:bbar}
b_{-a}=\rho_{a}^{\wt(b_{a})}\bbar _{a},\quad a\in\Ip,
\end{align}
for a formal $\rho_{a}$ (later  identified with a Schottky sewing parameter). 
Then $\{b_{-a}\}$ is a dual basis for the Li-Z metric $\langle \cdot,\cdot\rangle_{\rho_{a}}$ with adjoint \eqref{eq:udagger}
\begin{align}\label{RhoAdjoint}
u^{\dagger}_{\rho_{a}}(m)=(-1)^{N}\rho_{a}^{m-N+1}u(2N-2-m),
\end{align}
for $u$ quasiprimary  of weight $N$.
Let $\bm{b}_{+}=(b_{1},\ldots,b_{g})$ denote an element of a $V^{\otimes g}$-basis with   dual basis $\bm{b}_{-}=(b_{-1},\ldots,b_{-g})$ with respect to the Li-Z metric $\langle \cdot,\cdot\rangle_{\rho_{a}}$. Let $w_{a}$ for $a\in\I$ be $2g$ formal variables (later identified with the canonical Schottky parameters). Consider the genus zero $2g$-point correlation function
\begin{align*}
Z^{(0)}(\bm{b,w})=&Z^{(0)}(b_{-1},w_{-1};b_{1},w_{1};\ldots;b_{-g},w_{-g};b_{g},w_{g})
\\
=&\prod_{a\in\Ip}\rho_{a}^{\wt(b_{a})}Z^{(0)}(\bbar_{1},w_{-1};b_{1},w_{1};\ldots;\bbar_{g},w_{-g};b_{g},w_{g}).
\end{align*}
We can now define the genus $g$ partition function as
\begin{align}\label{GenusgPartition}
Z_{V}^{(g)}:=Z_{V}^{(g)}(\bm{w,\rho})
=\sum_{\bm{b}_{+}}Z^{(0)}(\bm{b,w}),
\end{align}
for $\bm{w,\rho}=w_{\pm 1},\rho_{1},\ldots ,w_{\pm g},\rho_{g}$ and 
where the sum is over any basis $\{\bm{b}_{+}\}$ of $V^{\otimes g}$.
This definition is motivated by the sewing relation \eqref{eq:SchottkySewing2} and ideas in~\cite{MT2,MT3}. 
\begin{remark} 
	Note that $Z_{V}^{(g)}$ depends on  $\rho_{a}$ via the dual vectors $\bm{b}_{-}$ as in \eqref{eq:bbar}. In particular, setting $\rho_{a}=0$ for some $a\in\Ip$ then $Z_{V}^{(g)}$ degenerates to a genus $g-1$ partition function. Furthermore, the genus $g$ partition function for the tensor product $V_{1}\otimes V_{2}$ of two VOAs $V_{1}$ and $V_{2}$ is  $Z^{(g)}_{V_{1}\otimes V_{2}}=Z_{V_{1}}^{(g)}Z_{V_{2}}^{(g)}$.
\end{remark}

We may  define the genus $g$  formal $n$-point function for $n$ vectors $v_{1},\ldots,v_{n}\in V$ inserted at $y_{1},\ldots,y_{n}$ by
\begin{align}\label{GenusgnPoint}
Z_{V}^{(g)}(\bm{v,y}):=Z_{V}^{(g)}(\bm{v,y};\bm{w,\rho})
=
\sum_{\bm{b}_{+}}Z^{(0)}(\bm{v,y};\bm{b,w}),
\end{align}
where $
Z^{(0)}(\bm{v,y};\bm{b,w})=Z^{(0)}(v_{1},y_{1};\ldots;v_{n},y_{n};b_{-1},w_{-1};\ldots;b_{g},w_{g})$. We show below in Theorem~\ref{theor:ZZhu} that this definition concurs with that of Zhu for genus one~\cite{Z}.

Let $U$ be a subVOA of $V$ where $V$ has a $U$-module decomposition 
$V=\bigoplus_{\alpha\in A}\mathcal{W}_{\alpha}$ for  $U$-modules $\mathcal{W}_{\alpha}$ and some indexing set $A$.
Let $\mathcal{W}_{\bm{\alpha}}=\bigotimes_{a=1}^{g} \mathcal{W}_{\alpha_{a}}$ denote a tensor product of $g$ modules labelled $\alpha_{1},\ldots ,\alpha_{g}$. We define
\begin{align}\label{eq:Z_Walpha}
Z_{\mathcal{W}_{\bm{\alpha}}}^{(g)}(\bm{v,y}) :=\sum _{\bm{b_{+}}\in \mathcal{W}_{\bm{\alpha}}} Z^{(0)}(\bm{v,y};\bm{b,w}),
\end{align}
where here the sum is over a basis $\{\bm{b}_{+}\}$ for  $\mathcal{W}_{\bm{\alpha}}$. It follows that
\begin{align}
\label{eq:Z_WalphaSum}
Z_{V}^{(g)}(\bm{v,y})=\sum_{\bm{\alpha}\in\bm{A}}Z_{\mathcal{W}_{\bm{\alpha}}}^{(g)}(\bm{v,y}),
\end{align}
where the sum ranges over $\bm{\alpha}=(\alpha_{1},\ldots ,\alpha_{g}) \in \bm{A}$ for $\bm{A}=A^{\otimes{g}}$. 
Finally, it is useful to define corresponding formal $n$-point correlation differential forms 
\begin{align*}
\F_{V}^{(g)}(\bm{v,y}):=Z^{(g)}(\bm{v,y})\bm{dy^{\wt(v)}},\quad 
\F_{\mathcal{W}_{\bm{\alpha}}}^{(g)}(\bm{v,y})=Z_{\mathcal{W}_{\bm{\alpha}}}^{(g)}(\bm{v,y})\bm{dy^{\wt(v)}},
\end{align*}
where  $\bm{dy^{\wt(v)}}=\prod_{k=1}^{n}dy_{k}^{\wt(v_{k})}$.

\subsection{Comparison to genus one trace functions} 
 We  now show  that we obtain Zhu's definition~\cite{Z} of a genus one $n$-point function (including the partition function\footnote{
 An alternative way to find  the partition function based on Catalan series  is described in~\cite{MT3}.})
 from our genus $g$ definition \eqref{GenusgnPoint} when $g=1$. 
In this case there are three Schottky parameters $W_{\pm 1}$ and $q=q_{1}$
with corresponding canonical coordinates $w_{\pm 1}$ and $\rho=\rho_{1}$. 
The Schottky group is $\Gamma=\langle \gamma \rangle $ where 
$\gamma=\sigma^{-1}
\left(
\begin{smallmatrix}
	q^{1/2} &0\\
	0 &q^{-1/2}
\end{smallmatrix}  \right)
\sigma$ for 
\begin{align}
\label{eq:sigma}
\sigma=(-W)^{-1/2}  \begin{pmatrix}1 & -W_{-1}\\1 & -W_{1}\end{pmatrix},
\end{align}
where $W=W_{1}-W_{-1}$ from \eqref{eq:gammaa}.
For $x\in\D$ then $\zeta =\log X$ where $X=\sigma x=\frac{x-W_{-1}}{x-W_{1}}$ lies in a fundamental region with periods $\tpi,\tpi \tau$  in the standard description of a torus with modular parameter $q=e^{\tpi \tau}$.

The genus one $n$-point function  for $v_{1},\ldots ,v_{n}$ is defined by Zhu by the trace~\cite{Z}
\begin{align}
\label{eq:Zhunpt}
Z^{\Zhu}_{V}(\bm{v,\zeta};q):=&
\Tr_{V} Y\left(e^{ \wt(v_{1})\zeta_{1}}v_{1},e^{\zeta_{1}}\right)\ldots  
Y\left(e^{\wt(v_{n})\zeta_{n} }v_{n},e^{\zeta_{n}}\right)q^{L(0)}  ,
\end{align}
where we have omitted the standard  $q^{-C/24}$ factor  introduced to enhance the $\SL_{2}(\Z)$ modular properties of trace functions.
Comparing to our definition \eqref{GenusgnPoint} we find
\begin{theorem}\label{theor:ZZhu}
Let $v_{1},\ldots ,v_{n}$ be quasiprimary vectors inserted at $x_{1},\ldots, x_{n}$. Then 
\[
\F_{V}^{(1)}(\bm{v,x})=Z^{\Zhu}_{V}(\bm{v,\zeta};q)\bm{d\zeta^{\wt(v)}},
\]
for  $\zeta_{i} =\log X_{i}$ with $X_{i}=\sigma x_{i}=\frac{x_{i}-W_{-1}}{x_{i}-W_{1}}$ for $i=1,\ldots, n$.
\end{theorem}
In order to prove Theorem~\ref{theor:ZZhu} we require the following  identities:
\begin{proposition}
The $\slLie_{2}(\C)$ generators $L(\pm 1),L(0)$ obey the relations:
\begin{align}
\lambda^{L(0)} L(\pm 1)\lambda^{-L(0)}&=\lambda^{\mp 1}L(\pm 1),
\label{eq:L0Lpm1conj}
\\[3pt]
e^{x L(\pm 1)}L(0)e^{-x L(\pm 1)}&=L(0)\pm xL(\pm 1),
\label{eq:LpmL0conj}
\\[3pt]
e^{x L(1)}L(-1)e^{-x L(1)}&=L(-1)+2xL(0)+x^{2}L(1),
\label{eq:L1Lm1conj}
\\[3pt]
e^{x L(-1)}e^{yL(1)}&=(1-xy)^{L(0)}e^{y L(1)} e^{x L(-1)}(1-xy)^{L(0)}.
\label{eq:Lm1L1com}
\end{align}
\end{proposition}
\begin{proof}
 \eqref{eq:L0Lpm1conj}-\eqref{eq:L1Lm1conj} are standard vertex algebra results as follows~\cite{FHL}.
\eqref{eq:L0Lpm1conj} holds since $\lambda^{L(0)}Y(\omega ,z)\lambda^{-L(0)}=Y(\lambda^{2}\omega ,\lambda z)$. Using $e^{xA}Be^{-xA}=e^{x\,\ad_{A}}B$ for linear operators $A,B$ then \eqref{eq:LpmL0conj}  and \eqref{eq:L1Lm1conj} follow. 
Let $K(x,y)$ denote the right hand side of \eqref{eq:Lm1L1com}. 
Using  \eqref{eq:L0Lpm1conj}-\eqref{eq:L1Lm1conj} 
we find  $\del_{x}K(x,y)=L(-1)K(x,y)$ 
so that $K(x,y)=e^{xL(-1)}K(0,y)=e^{x L(-1)}e^{yL(1)}$.

An  elementary Lie algebra proof follows by verifying  \eqref{eq:L0Lpm1conj}-\eqref{eq:Lm1L1com} for the fundamental $\slLie(2,\C)$ Lie algebra representation with the canonical identifications:
$L(-1)=\left(\begin{smallmatrix}
0 & 1 \\ 0 & 0\end{smallmatrix}\right)$, 
$L(0)=\left(\begin{smallmatrix}
1 & 0 \\ 0 & -1 \end{smallmatrix}\right)$
 and $L(1)=\left(\begin{smallmatrix}
0 &  0 \\ -1 & 0\end{smallmatrix}\right)$
and  $\SL_{2}(\C)$  elements 
$e^{x L(-1)}=\left(\begin{smallmatrix}
	1 & x \\ 0 & 1\end{smallmatrix}\right)$, 
$e^{yL(1)}=\left(\begin{smallmatrix}
	1&  0 \\ -y & 1\end{smallmatrix}\right)$
and 
$
\lambda^{L(0)} = \left(\begin{smallmatrix}
\sqrt{\lambda}  & 0 \\ 0 & 1/\sqrt{\lambda}
\end{smallmatrix}\right)$.
 \eqref{eq:L0Lpm1conj}-\eqref{eq:Lm1L1com} therefore hold for any $\slLie(2,\C)$  representation.  
\end{proof}
\begin{remark}
Using  \eqref{eq:L0Lpm1conj} then identity \eqref{eq:Lm1L1com} is equivalent to  
\begin{align}
\label{eq:Bigid}
e^{x L(-1)}e^{yL(1)}e^{-x(1-xy)^{-1} L(-1)} 
e^{-y(1-xy)L(1)}(1-xy)^{-2L(0)}
=\Id_{V}.
\end{align}
\end{remark}

\begin{proof}[Proof of Theorem \ref{theor:ZZhu}]
Consider  
\begin{align*}
\F_{V}^{(1)}(\bm{v,x})
=&\sum_{b_{1}\in V}\left\langle 
\vac,Y\left(b_{1},w_{1}\right)\bm{ Y\left(v,x\right)}Y\left(b_{-1},w_{-1}\right)\vac\right\rangle_{\rho}\bm{dx^{\wt(v)}}
\\
=&\sum_{n\ge 0}\rho^{n}\sum_{b\in V_{n}}\left\langle 
\vac,Y\left(b,w_{1}\right)\bm{Y(v,x)}Y\left(\bbar,w_{-1}\right)\vac\right\rangle \bm{dx^{\wt(v)}},
\end{align*}
recalling \eqref{eq:uvA} for standard Li-Z metric with $\rho=1$.
For $\sigma$ of \eqref{eq:sigma} we note that  
\begin{align*}
&\sigma w_{-1}=q,\quad \sigma w_{1}=\frac{1}{q},
\\
& c(cw_{-1}+d)=\frac{1}{1-q},\quad  c(cw_{1}+d)=-\frac{q}{1-q},
\\
& (cw_{-1}+d)(cw_{1}+d)=\frac{qW}{(1-q)^{2}},
\end{align*}
where $c=(-W)^{-1/2}$ and $d=   -(-W)^{-1/2}W_{1}$.
Applying \eqref{eq:MobV} for $\sigma $ and  recalling from \eqref{eq:rhoa} that $\rho =-q W^{2}/(1-q)^{2}$  we obtain
\begin{align*}
\F_{V}^{(1)}(\bm{v,x})
&=\sum_{n\ge 0}
\left(-q^{-1}(1-q)^2\right)^{n}
\sum_{b\in V_{n}}
 \left\langle 
\vac,Y\left(e^{\frac{q}{1-q}L(1)}b,q^{-1}\right)
\bm{ Y\left(v,X\right)}Y\left(e^{-\frac{1}{1-q}L(1)}\bbar,q\right)\vac\right\rangle 
\bm{dX^{\wt(v)}},
\end{align*}
for $X_{i}=\sigma x_{i}$.
From \eqref{eq:Yadj} we find $Y\left(e^{\frac{q}{1-q}L(1)}b,q^{-1}\right) 
= Y^{\dagger}\left((-q^2)^{L(0)}e^{\frac{q^2}{1-q}L(1)}b,q\right)$ so that 
\begin{align*}
& \left\langle 
\vac,\;Y\left(e^{\frac{q}{1-q}L(1)}b,q^{-1}\right)
\bm{ Y\left(v,X\right)}Y\left(e^{-\frac{1}{1-q}L(1)}\bbar,q\right)\vac\right\rangle 
\\
&=
\left\langle 
Y\left((-q^2)^{L(0)}e^{\frac{q^2}{1-q}L(1)}b,q\right)\vac,\;
\bm{ Y\left(v,X\right)}Y\left(e^{-\frac{1}{1-q}L(1)}\bbar,q\right)\vac\right\rangle 
\\
&=
\left\langle 
e^{qL(-1)}(-q^2)^{L(0)}e^{\frac{q^2}{1-q}L(1)}b,\;
\bm{ Y\left(v,X\right)}e^{qL(-1)}e^{-\frac{1}{1-q}L(1)}\bbar\right\rangle 
\\
&=
\left\langle 
b,\;
e^{\frac{q^2}{1-q}L(-1)}(-q^2)^{L(0)}e^{qL(1)}
\bm{ Y\left(v,X\right)}e^{qL(-1)}e^{-\frac{1}{1-q}L(1)}\bbar\right\rangle .
\end{align*}
Thus it follows that 
\begin{align*}
\F_{V}^{(1)}(\bm{v,x})
&=
\Tr_{V}\left(
 e^{\frac{q^2}{1-q}L(-1)}(-q^2)^{L(0)}e^{qL(1)}\bm{Y\left(v,X\right)}e^{qL(-1)}e^{-\frac{1}{1-q}L(1)}\left(-q^{-1}(1-q)^2\right)^{L(0)}\right)
\bm{dX^{\wt(v)}}.
\end{align*}
Hence using $\Tr (AB)=\Tr (BA)$  we obtain
\begin{align*}
\F_{V}^{(1)}(\bm{v,x})
&= \Tr_{V}\left(\bm{Y\left(v,X\right)}
e^{qL(-1)}e^{-\frac{1}{1-q}L(1)}
\left(-q^{-1}(1-q)^2\right)^{L(0)}
e^{\frac{q^2}{1-q}L(-1)}(-q^2)^{L(0)}e^{qL(1)}\right)\bm{dX^{\wt(v)}}
\\
&= \Tr_{V}\left(\bm{Y\left(v,X\right)}
e^{qL(-1)}e^{-\frac{1}{1-q}L(1)}
e^{ -q(1-q)L(-1)}e^{\frac{1}{(1-q)^2}L(1)}
\left(q(1-q)^2\right)^{L(0)}\right)\bm{dX^{\wt(v)}},
\end{align*}
using \eqref{eq:L0Lpm1conj}.
Letting $x=q,y=-(1-q)^{-1}$ so that $1-xy=(1-q)^{-1}$ in \eqref{eq:Bigid} we find
\[
e^{qL(-1)}e^{-\frac{1}{1-q}L(1)}e^{ -q(1-q)L(-1)}e^{\frac{1}{(1-q)^2}L(1)}\left(1-q\right)^{2L(0)}=\Id_{V}.
\]
Thus we have shown that
\begin{align*}
\F_{V}^{(1)}(\bm{v,x})
&=
\Tr_{V}\left( \bm{ Y\left(v,X\right)}q^{L(0)}\right) \bm{dX^{\wt(v)}}
\\
&=Z^{\Zhu}_{V}(\bm{v,\zeta};q)\bm{d\zeta^{\wt(v)}},
\end{align*}
since $\bm{dX^{\wt(v)}}=\prod_{k=1}^{n}e^{\wt(v_{k})\zeta_{k}}d\zeta_{k}^{\wt(v_{k})}$. 
Therefore  the theorem holds.
\end{proof}

\subsection{$\SL_{2}(\C)$-invariance}
For $W_{a}$ of \eqref{eq:SchottkySewing} define the differential operators\footnote{$\mathcal{L}_{\pm 1},\mathcal{L}_{0}$  are elements of the tangent space $T(\mathfrak{C}_{g})$.}
\begin{align}
\mathcal{L}_{r}:=-\sum_{a\in\I}W_{a}^{r+1}\del_{W_{a}},\quad r=0,\pm 1,
\label{eq:ln}
\end{align}
satisfying the  $\slLie(2,\C)$  algebra $[\mathcal{L}_{r},\mathcal{L}_{s}]=(r-s)\mathcal{L}_{r+s}$. From \eqref{eq:wa} and \eqref{eq:rhoa} we find
	\begin{align*}
	\mathcal{L}_{-1} &=-\sum_{a\in\I}\del_{w_{a}},\quad 
	\mathcal{L}_{0}=-\sum_{a\in\I}\left(w_{a}\del_{w_{a}}+\rho_{a}\del_{\rho_{a}}\right),
	\notag
	\\
	\mathcal{L}_{1}&=-\sum_{a\in\I}\left((w_{a}^{2}+\rho_{a})\del_{w_{a}}+2w_{a}\rho_{a}\del_{\rho_{a}}\right).
	\end{align*}
\begin{proposition}\label{prop:ZSL2C}
	The genus $g$ partition function is  $\SL_{2}(\C)$-invariant with
	\begin{align*}
	\mathcal{L}_{r}Z_{V}^{(g)}=0,\quad r=-1,0,1.
	\end{align*}
\end{proposition}
\begin{proof}
	Let $\ell=0$ and $u =\omega$ in~\eqref{ZhuPrepLemma} to obtain
	\begin{align*}
	0=\sum_{a\in\I}Z^{(0)}(\ldots;L(-1)b_{a},w_{a};\ldots)=\sum_{a\in\I}\del_{w_{a}}Z^{(0)}(\bm{b,w}),
	\end{align*}
	using  the VOA translation property. Summing over a $V^{\otimes g}$-basis we obtain $\mathcal{L}_{-1}Z_{V}^{(g)}=0$.
	Similarly, with $\ell=1$  we find \eqref{ZhuPrepLemma} gives
	\begin{align*}
	0&=\sum_{a\in\I}\left(  w_{a}Z^{(0)}(\ldots;L(-1)b_{a},w_{a};\ldots)+Z^{(0)}(\ldots;L(0)b_{a},w_{a};\ldots)
	\right)
	\\
	&=\sum_{a\in\I}\left(  w_{a}\del_{w_{a}}+\wt(b_{a}
	\right)Z^{(0)}(\bm{b,w})
	=\sum_{a\in\I}\left(w_{a}\del_{w_{a}}+\rho_{a}\del_{\rho_{a}}\right)Z^{(0)}(\bm{b,w}),
	\end{align*}
	recalling \eqref{eq:bbar}. 
	Thus  on summing over a $V^{\otimes g}$-basis we find $\mathcal{L}_{0}Z_{V}^{(g)}=0$.
	Lastly, for $\ell=2$   we find \eqref{ZhuPrepLemma} implies
	\begin{align*}
	0=&\sum_{a\in\I}\left(
	w_{a}^{2}Z^{(0)}(\ldots;L(-1)b_{a},w_{a};\ldots)+2w_{a}Z^{(0)}(\ldots;L(0)b_{a},w_{a};\ldots)\right.\\&\left.\quad\quad+Z^{(0)}(\ldots;L(1)b_{a},w_{a};\ldots)
	\right),
	\\
	=&\sum_{a\in\I}\left(
	\left(
	w_{a}^{2}\del_{w_{a}}+2w_{a}\rho_{a}\del_{\rho_{a}}\right)Z^{(0)}(\bm{b,w})
	+Z^{(0)}(\ldots;L(1)b_{a},w_{a};\ldots)
	\right).
	\end{align*}
	Summing $b_{a}$ over a $V$-basis and using  Lemma~\ref{AdjointLemma} with  $L^{\dagger}_{\rho_{a}}(1)=\rho_{a}L(-1)$  gives
	\begin{align*}
	&\sum_{{b}_{a}}Z^{(0)}(\ldots;L(1)b_{a},w_{a};\ldots)
	=\sum_{{b}_{a}}\rho_{a}Z^{(0)}(\ldots;L(-1)b_{-a},w_{-a};\ldots)
	\\
	&
	=\rho_{a}\del_{w_{-a}}\sum_{{b}_{a}}Z^{(0)}(\bm{b,w}).
	\end{align*}
	Altogether summing over a $V^{\otimes g}$-basis we obtain $\mathcal{L}_{1}Z_{V}^{(g)}=0$.
\end{proof}
By a similar analysis we find in general that for $n$-point functions:
\begin{proposition}\label{prop:nptSL2CProp}
	$Z_{V}^{(g)}(\bm{v,y})$ is $\SL_{2}(\C)$ covariant as follows:
	\begin{enumerate}[itemsep=6pt]
		\item \quad 
		$\Big(\mathcal{L}_{-1}+	\mathcal{L}_{-1}^{\bm{y}}\Big)Z_{V}^{(g)}(\bm{v,y})
		=0$,
		\item \quad
		$	\Big(
		\mathcal{L}_{0}+\mathcal{L}_{0}^{\bm{y}}-	\sum_{k=1}^{n}\wt(v_{k}) 
		\Big)
		Z_{V}^{(g)}(\bm{v,y})=0$,
		\item \quad
		$\Big(
		\mathcal{L}_{1}+\mathcal{L}_{1}^{\bm{y}}-\sum_{k=1}^{n}2y_{k}\wt(v_{k}) 
	\Big)Z_{V}^{(g)}(\bm{v,y})=
		 	\sum_{k=1}^{n}Z_{V}^{(g)}(\ldots ;L(1)v_{k},y_{k};\ldots)$,
	\end{enumerate}
	\medskip
	\noindent 
	where $\mathcal{L}_{r}^{\bm{y}}=-\sum_{k=1}^{n}y_{k}^{r+1}\del_{{k}}$ for $r=0,\pm 1$.
\end{proposition}

\begin{example}
For  $g=1$ for quasiprimary $v_{1},\ldots ,v_{n}$ then Theorem~\ref{theor:ZZhu} implies
\begin{align*}
Z_{V}^{(1)}(\bm{v,y})=Z^{\Zhu}_{V}(\bm{v,\zeta};q)\bm{ \zeta'^{\wt(v)}},
\end{align*}
where $ \bm{ \zeta'^{\,\wt(v)}}
=\prod_{k=1}^{n}\zeta_{k}'^{\,\wt(v_{k})}$ for $\zeta_{k}'= \frac{1}{y_{k}-W_{-1}}-\frac{1}{y_{k}-W_{1}} $.
It is easy to check Proposition~\ref{prop:nptSL2CProp}~(1) and (2) hold for $Z_{V}^{(1)}(\bm{v,y})$. Part (3) follows from the identities
\begin{align*}
\left(
\mathcal{L}_{1}+\mathcal{L}_{1}^{\bm{y}}
\right)\zeta_{k} =W_{1}-W_{-1},\quad 
\left(\mathcal{L}_{1}+\mathcal{L}_{1}^{\bm{y}}
\right)\zeta_{k}'=2y_{k}\zeta_{k}',
\end{align*} 
together with genus one translation symmetry  where  from \eqref{eq:Zhunpt} and for fixed  $\alpha$
\begin{align*}
Z^{\Zhu}_{V}(\ldots,\zeta_{k}+\alpha,v_{k},\ldots;q)=&
\Tr_{V} Y\left(
e^{ \wt(v_{1})(\zeta_{1}+\alpha)}v_{1},e^{\zeta_{1}+\alpha}\right)
\ldots  
Y\left(e^{ \wt(v_{n})(\zeta_{n}+\alpha)}v_{n},e^{\zeta_{n}+\alpha}\right)q^{L(0)}
\\
=&
 \Tr_{V} e^{\alpha L(0)}Y\left(e^{ \wt(v_{1})\zeta_{1}}v_{1},e^{\zeta_{1}}\right)\ldots  
 Y\left(e^{\wt(v_{n})\zeta_{n} }v_{n},e^{\zeta_{n}}\right)q^{L(0)}e^{-\alpha L(0)}
	\\
=	&
Z^{\Zhu}_{V}(\bm{v,\zeta};q).
\end{align*}
\end{example}	

\begin{remark}\label{rem:sl2C}
We may encapsulate Proposition~\ref{prop:nptSL2CProp}~(1)--(3) into one expression as follows. Let $p(z)\in \Pi_{2}(z)$, the space of degree $2$ polynomials in $z$, and define  
\begin{align}\label{eq:calLdef}
\mathcal{L}_{p}^{\Cg}:=-\sum_{a\in\I}p(W_{a})\del_{W_{a}},
\quad 
\mathcal{L}_{p}^{\Cgn}:=\mathcal{L}_{p}^{\Cg}-\sum_{k=1}^{n}p(y_{k})\del_{{k}}.
\end{align}
Clearly $\mathcal{L}_{z^{r+1}}^{\Cg}=\mathcal{L}_{r}$ and $\mathcal{L}_{z^{r+1}}^{\Cgn}=\mathcal{L}_{r}+\mathcal{L}_{r}^{\bm{y}}$ for $r=-1,0,1$. Thus $\mathcal{L}_{p}^{\Cg}$ and $\mathcal{L}_{p}^{\Cgn}$ generate $\slLie_{2}(\C)$ and Proposition~\ref{prop:nptSL2CProp}~(1)--(3) can be rewritten as
\begin{align}\label{eq:Lpnpt}
\left(\mathcal{L}_{p}^{\Cgn}-\sum_{k=1}^{n}
\del_{{k}} p(y_{k})\,\wt(v_{k}) \right)Z_{V}^{(g)}(\bm{v,y})
=\del_{{k}}^{(2)} p (y_{k})Z_{V}^{(g)}(\ldots ;L(1)v_{k},y_{k};\ldots).
\end{align}
\end{remark}

\subsection{Genus $g$ Zhu recursion}
Consider $Z_{V}^{(g)}(u,x;\bm{v,y})$ for quasiprimary $u$ of weight $N$. We derive a Zhu reduction formula given in Theorem~\ref{theor:ZhuGenusg}  generalizing Theorem~\ref{theor:GenusZeroZhu} to genus $g$. 
We firstly apply Theorem~\ref{theor:GenusZeroZhu} to  the summands of \eqref{GenusgnPoint} to find
\begin{align}\label{GenusgZhu1}
Z_{V}^{(g)}(u,x;\bm{v,y})\notag
&=\sum_{a\in\I}\sum_{j\ge 0}\del^{(0,j)}\psi_{N}^{(0)}(x,w_{a})\sum_{\bm{b}_{+}}Z^{(0)}(\ldots ;u(j)b_{a},w_{a};\ldots)\notag
\\
&\quad+\sum_{k=1}^{n}\sum_{j\ge 0}\del^{(0,j)}\psi_{N}^{(0)}(x,y_{k})\sum_{\bm{b}_{+}}Z^{(0)}(\ldots ;u(j)v_{k},y_{k};\ldots).
\end{align}
Define a column vector $X=(X_{a}(m))$ indexed by $ m\ge 0$ and $ a\in\I$  with components 
\begin{align}\label{XamDef}
X_{a}(m):=\rho_{a}^{-\frac{m}{2}}\sum_{\bm{b}_{+}}Z^{(0)}(\ldots;u(m)b_{a},w_{a};\ldots),
\end{align}
and  a row vector $p(x)=(p_{a}(x,m))$  for $m\ge 0, a\in\I$  with components 
\begin{align}\label{eq:pdef}
p_{a}(x,m):=\rho_{a}^{\frac{m}{2}}\del^{(0,m)}\psi_{N}^{(0)}(x,w_{a}).
\end{align}
Then the genus zero Zhu recursion formula~\eqref{GenusgZhu1} can be rewritten as
\begin{align}\label{ZhuGenusZero}
Z_{V}^{(g)}(u,x;\bm{v,y})=p(x)X+\sum_{k=1}^{n}\sum_{j\ge 0}\del^{(0,j)}\psi_{N}^{(0)}(x,y_{k})Z_{V}^{(g)}(\ldots;u(j)v_{k},y_{k};\ldots).
\end{align}
We next develop a recursive formula for $X$ following the genus two  approach in ~\cite{GT}.
We apply Lemma~\ref{AdjointLemma} to the summands of $X_{a}(m)$ 
with adjoint \eqref{RhoAdjoint} to find
\begin{align}\label{XaAdjoint}
X_{a}(m)=(-1)^{N}\rho_{a}^{\frac{m}{2}-N+1}\sum_{\bm{b}_{+}}Z^{(0)}(\ldots;u(2N-2-m)b_{-a},w_{-a};\ldots).
\end{align}
In particular, \eqref{XaAdjoint} implies that for $0\le \ell \le 2N-2$ then 
\begin{align}\label{XaRelationGamma}
X_{a}(\ell)=(-1)^{N}X_{-a}(2N-2-\ell).
\end{align}
Let $t=m-2N+1\ge 0$ and use Theorem~\ref{GenusZeroZhu3} and Remark~\ref{rem:Zhu genus zero II} so that \eqref{XaAdjoint} implies
\begin{align*}
X_{a}(m)&=(-1)^{N}\rho_{a}^{\frac{t+1}{2}}\sum_{\bm{b}_{+}}Z^{(0)}(\ldots;u(-t-1)b_{-a},w_{-a};\ldots)
\\
&=(-1)^{N}\rho_{a}^{\frac{t+1}{2}} \sum_{k=1}^{n}\sum_{j\ge 0}\partial^{(t,j)}\psi_{N}^{(0)}(w_{-a},y_{k})Z^{(0)}(\ldots;u(j)v_{k},y_{k};\ldots)
\\
&\quad+(-1)^{N}\rho_{a}^{\frac{t+1}{2}} \sum_{\substack{b\in\I,\\b\neq -a}}\sum_{j\ge 0}\partial^{(t,j)}\psi_{N}^{(0)}(w_{-a},w_{b})Z^{(g)}_{V}(\ldots;u(j)b_{b},w_{b};\ldots)
\\
&\quad+(-1)^{N}\rho_{a}^{\frac{t+1}{2}} \sum_{j\ge 0}\E_{t}^{j}(w_{-a})Z^{(g)}_{V}(\ldots;u(j)b_{-a},w_{-a};\ldots).
\end{align*}
We can rewrite this as
\begin{align}\label{XaRelation}
X_{a}(m)=(G+RX)_{a}(m-2N+1),\quad m\ge 2N-1, 
\end{align}
with column vector $G=(G_{a}(m))$ for $m\ge 0, a\in\I$ given by 
\begin{align*}
G:=\sum_{k=1}^{n}\sum_{j\ge 0}\del_{k}^{(j)}q(y_{k})Z_{V}^{(g)}(\ldots;u(j)v_{k},y_{k};\ldots),
\end{align*}
where $q(y)=(q_{a}(y;m))$ for $m\ge 0, a\in\I$ is a column vector with components 
\begin{align}
\label{eq:qdef}
q_{a}(y;m):=(-1)^{N}\rho_{a}^{\frac{m+1}{2}}\del^{(m,0)}\psi_{N}^{(0)}(w_{-a},y),
\end{align}
and $R=(R_{ab}(m,n))$ for $m,n\ge 0$ and $ a,b\in\I$ is a doubly indexed matrix with components 
\begin{align}
R_{ab}(m,n):=\begin{cases}(-1)^{N}\rho_{a}^{\frac{m+1}{2}}\rho_{b}^{\frac{n}{2}}\del^{(m,n)}\psi_{N}^{(0)}(w_{-a},w_{b}),&a\neq-b,\\ 
(-1)^{N}\rho_{a}^{\frac{m+n+1}{2}}\E_{m}^{n}(w_{-a}),&a=-b.
\end{cases}
\label{eq:Rdef}
\end{align}
We note from \eqref{eq:Ejt} that $R_{ab}(m,n)=0$ for $a=-b$ for all $n\ge 2N-1$.

Define the doubly indexed matrix $\Delta=(\Delta_{ab}(m,n))$ by
\begin{align}
\Delta_{ab}(m,n):=\delta_{m,n+2N-1}\delta_{ab}.\label {eq:Deltadef}
\end{align}
Note that $\Delta$ satisfies the identities
\begin{align}\label{DeltaDelta}
\Delta^{T}\Delta =I,\quad  
\Delta\Delta^{T}=I-\Pi,
\end{align}
for projection matrix $\Pi=(\Pi_{ab}(m,n))$ with components
\begin{align*}
\Pi_{ab}(m,n):=\begin{cases}
\delta_{ab}\delta_{mn}& \mbox{ for } 0\le m,n\le 2N-2,\\
0 & \mbox{ otherwise.}
\end{cases}
\end{align*}
Equation~\eqref{XaRelation} can thus be written
\begin{align*}
X_{a}(m)=(\Delta(G+RX))_{a}(m),\quad m\ge 2N-1.
\end{align*}
Next  define
\begin{align*}
X^{\perp}:=\Delta^{T}X,\quad  
X^{\Pi}:=\Pi X.
\end{align*}
Using~\eqref{DeltaDelta} we have
\begin{align*}
X^{\perp}&=G+RX.
\end{align*}
We can use the relation
\begin{align*}
X=(\Pi+(I-\Pi))X=X^{\Pi}+\Delta X^{\perp},
\end{align*}
to obtain
\begin{align*}
X^{\perp}&=G+R\left(X^{\Pi}+\Delta X^{\perp}\right) =G+RX^{\Pi}+\widetilde{R}X^{\perp},
\end{align*}where $\widetilde{R}=R\Delta$. Formally solving for $X^{\perp}$ gives:
\begin{align*}
X^{\perp}=\left(I-\widetilde{R}\right)^{-1}RX^{\Pi}+\left(I-\widetilde{R}\right)^{-1}G,
\end{align*}
where the formal inverse $(I-\widetilde{R})^{-1}$ is given by 
\begin{align}
\label{eq:ImRinverse}
\left(I-\widetilde{R}\right)^{-1}=\sum_{k\ge 0}\widetilde{R}^{\,k}.
\end{align}
Altogether we have found 
\begin{lemma}\label{XSub}
Let $u$ be a quasiprimary vector with $\wt(u) = N$. Then
\begin{align*}
X=X^{\Pi}+\Delta \left(I-\widetilde{R}\right)^{-1}RX^{\Pi}+\Delta\left(I-\widetilde{R}\right)^{-1}G.
\end{align*}
\end{lemma}

\medskip
Recalling the preliminary Zhu recursion formula~\eqref{ZhuGenusZero}, we can substitute for $X$ using Lemma~\ref{XSub} to obtain the recursion formula 
\begin{align}\label{GenusgZhu}
&Z_{V}^{(g)}(u,x;\bm{v,y})=\chi(x)o(u;\bm{v,y})+\sum_{k=1}^{n}\sum_{j\ge 0}\del^{(0,j)}\psi_{N}(x,y_{k})Z_{V}^{(g)}(\ldots;u(j)v_{k},y_{k};\ldots),
\end{align}
where $\chi(x)=(\chi_{a}(x;\ell))$ and $o(u;\bm{v,y})=(o_{a}(u;\bm{v,y};\ell))$ are \emph{finite} row and column vectors indexed by 
$a\in\I$, $0\le \ell\le 2N-2$ with
\begin{align}
\label{eq:chiadef}
\chi_{a}(x;\ell)&:=\rho_{a}^{-\frac{\ell}{2}}(p(x)+\widetilde{p}(x)(I-\widetilde{R})^{-1}R)_{a}(\ell),
\\
\label{LittleODef}
o_{a}(\ell)&:=o_{a}(u;\bm{v,y};\ell)=\rho_{a}^{\frac{\ell}{2}}X_{a}(\ell),
\end{align}
and where $\widetilde{p}(x)=p(x)\Delta$ and $\psi_{N}(x,y)$ is defined by
\begin{align}
\label{eq:psilittleN}
\psi_{N}(x,y):=\psi_{N}^{(0)}(x,y)+\widetilde{p}(x)(I-\widetilde{R})^{-1}q(y).
\end{align}

The relation \eqref{XaRelationGamma} implies that $
o_{a}(\ell) = (-1)^{N} \rho_{a}^{\ell+1-N}o_{-a}(2N-2-\ell)$ so that  
\begin{align*}
\chi(x)o(u;\bm{v,y})= \sum_{a=1}^{g} \theta_{a}(x)o_{a}(u;\bm{v,y}),
\end{align*}
where for each $a \in\Ip$ we define a vector  $\theta_{a}(x)=(\theta_{a}(x;\ell) )$ indexed by $0\le \ell\le  2N-2$ with components
\begin{align}\label{eq:thetadef}
\theta_{a}(x;\ell) := \chi_{a}(x;\ell)+(-1)^{N }\rho_{a}^{N-1-\ell}\chi_{-a}(x;2N-2-\ell).
\end{align}

Now define the following vectors of formal differential forms
\begin{align}
\label{eq:ThetaPQdef}
\quad P(x):=p(x)dx^{N},\quad Q(y):=q(y)dy^{1-N},
\end{align}
with $\widetilde{P}(x)=P(x)\Delta$. Then with $\Psi_{N} (x,y):=\psi_{N}(x,y)dx^{N}dy^{1-N}$ we have
\begin{align}\label{GenusgPsiDef}
\Psi_{N}(x,y)=\Psi_{N}^{(0)}(x,y)+\widetilde{P}(x)(I-\widetilde{R})^{-1}Q(y).
\end{align}
Defining   $\Theta_{a}(x;\ell):=\theta_{a}(x;\ell)dx^{N}$  and $O_{a}(u; \bm{v,y};\ell) := o_{a}(u; \bm{v,y};\ell) \bm{dy^{\wt(v)}
	}$ we obtain our main result:
\begin{theorem}[Quasiprimary Genus $g$ Zhu Recursion]\label{theor:ZhuGenusg}
The genus $g$ $n$-point formal differential for a quasiprimary vector $u$ of weight $\wt(u)=N$ inserted at $x $ and general vectors $v_{1},\ldots,v_{n}$ inserted at $y_{1},\ldots,y_{n} $ respectively, satisfies the recursive identity
\begin{align}
\label{eq:ZhuGenusg}
\F_{V}^{(g)}(u,x;\bm{v,y})&=\sum_{a=1}^{g} \Theta_{a}(x)O_{a}(u;\bm{v,y})
\notag
\\
&\quad+\sum_{k=1}^{n}\sum_{j\ge 0}\del^{(0,j)}\Psi_{N}(x,y_{k})dy_{k}^{j}\,\F_{V}^{(g)}(\ldots;u(j)v_{k},y_{k};\ldots).
\end{align}
\end{theorem}
\begin{remark}
	\label{rem:MainTheorem}
The $\Theta_{a}(x),\Psi_{N}(x,y )$ coefficients depend on $N=\wt(u)$ but  are otherwise independent of the VOA $V$. As such, they are the analogue of the genus zero $\Psi_{N}^{(0)}(x,y )$ coefficients and the genus one Weierstrass coefficients found in~\cite{Z}. We note that for a 1-point function, \eqref{eq:ZhuGenusg} implies
\begin{align}
\label{eq:1ptfun}
\F_{V}^{(g)}(u,x)&=\sum_{a=1}^{g} \Theta_{a}(x)O_{a}(u).
\end{align}
\end{remark}

Much as for Corollary~\ref{cor:GenGenusZeroZhu} we may generalise Theorem~\ref{theor:ZhuGenusg} to find:
\begin{corollary}[General Genus $g$ Zhu Recursion]\label{cor:GenGenusgZhu}
The genus $g$ formal $n$-point  differential  for a quasiprimary  descendant $\tfrac{1}{i!}L(-1)^{i}u$  of $u$ of weight $N$ inserted at $x $ and general vectors $v_{1},\ldots,v_{n}$ inserted at $y_{1},\ldots,y_{n} $ respectively, satisfies the recursive identity
\begin{align*}
&\F_{V}^{(g)}(\tfrac{1}{i!}L(-1)^{i}u,x;\bm{v,y})
\notag
\\
&=
\sum_{a=1}^{g} \partial^{(i)}\Theta_{a}(x)dx^{i}\,O_{a}(u;\bm{v,y})+\sum_{k=1}^{n}\sum_{j\ge 0}\del^{(i,j)}\Psi_{N}(x,y_{k})\,dx^{i}\,dy_{k}^{j}\F_{V}^{(g)}(\ldots;u(j)v_{k},y_{k};\ldots).
\end{align*}
\end{corollary}
Lastly, we may describe a version of Theorem~\ref{theor:ZhuGenusg} for $U$-modules of a subVOA $U\subset V$. Let $V=\bigoplus_{\alpha\in A}\mathcal{W}_{\alpha}$ be the $U$-module decomposition  for  modules $\mathcal{W}_{\alpha}$ with some indexing set $A$. Let $\mathcal{W}_{\bm{\alpha}}=\bigotimes_{a=1}^{g} \mathcal{W}_{\alpha_{a}}$ denote a tensor product  as in \eqref{eq:Z_Walpha} of $g$ of these modules. Then by Remark~\ref{rem:adjoint} we find
\begin{theorem}\label{theor:ZhuGenusgModules}
The  $(n+1)$-point differential form $\F_{\mathcal{W}_{\bm{\alpha}}}^{(g)}(u,x;\bm{v,y})$   for a quasiprimary vector $u\in U$ of weight $N$ inserted at $x $ and $v_{1},\ldots,v_{n}\in V$ inserted at $y_{1},\ldots,y_{n} $ respectively, satisfies the recursive identity
\begin{align}\label{ZhuGenusgModules}
\F_{\mathcal{W}_{\bm{\alpha}}}^{(g)}(u,x;\bm{v,y})&
=\sum_{a=1}^{g} \Theta_{a}(x)O^{\mathcal{W}_{\bm{\alpha}}}_{a}(u;\bm{v,y})
\notag
\\
&\quad+\sum_{k=1}^{n}\sum_{j\ge 0}\del^{(0,j)}\Psi_{N}(x,y_{k})\,dy_{k}^{j} \F_{\mathcal{W}_{\bm{\alpha}}}^{(g)}(\ldots;u(j)v_{k},y_{k};\ldots).
\end{align}
\end{theorem}
A genus $g$ version of Theorem~\ref{GenusZeroZhu3} can also be developed.

\section{Convergence of $\Psi_{N}(x,y)$ and $\Theta_{a}(x;\ell)$}\label{PsiPhiConvergence}
\subsection{A Poincar\'{e} sum for $\Psi_{N}(x,y)$}\label{SSPoincareSum}
In this section we rewrite $\Psi_{N}(x,y)$ as a Poincar\'{e} sum over the Schottky group $\gamma\in \Gamma$ of \S\ref{sec:Schottky} with summand $\Psi_{N}^{(0)}(\gamma x,y)$. 
The convergence of this sum depends on the choice of  $f_{\ell}(x)$ terms in  $\psi_{N}^{(0)}(x,y)$. 
For $N=1$ we show that we may choose $\psi_{1}^{(0)}(x,y)$ so that $ \Psi_{1}(x,y)$ is a classical differential of the third kind 
 and each $\Theta_{a}(x;0)$ is a holomorphic 1-form. 
For $N\ge 2$, for 
a choice of $\psi_{N}^{(0)}(x,y)$ due to Bers~\cite{Be1}, the Poincar\'{e} sum  is holomorphic for $x\neq y$ for all  $x,y\in\D=\Chat-\cup_{a\in\I}\Delta_{a}$, 
the fundamental Schottky domain together with identified boundaries $\gamma_{a}\calC_{a}=-\calC_{-a}$ where  $\Delta_{a}$ is the  disc with boundary $\calC_{a}$.  
Furthermore $\{\Theta_{a}(x;\ell)\}$ spans the space of holomorphic $N$-forms. 
We also describe a canonical choice giving a normalized basis of holomorphic $N$-forms. 
    
 Using \eqref{eq:pdef} and  \eqref{eq:Rdef} we firstly note that  $\widetilde{p}(x)=p(x)\Delta$ and $\widetilde{R}=R\Delta $ are independent of the choice of $\psi_{N}^{(0)}(x,y)$ and are given by 
\begin{align}
\label{eq:ptilde}
\widetilde{p}_{b}(x;n)&=\frac{\rho_{b}^{\frac{n+2N-1}{2}}}{(x-w_{b})^{n+2N}},
\\
\label{eq:Rtilde}
\widetilde{R}_{ab}(m,n)&=\begin{cases}(-1)^{N}\rho_{a}^{\frac{m+1}{2}}
\del^{(m)}\widetilde{p}_{b}(w_{-a};n),&a\neq-b\\ 
0,&a=-b.
\end{cases}
\end{align}
Recall the Schottky group generator $\gamma_{a}$  for  $a\in\I$ of \eqref{eq:gamma_a.z} for which
\begin{align*}
\gamma_{a}x=w_{-a}+\frac{\rho_{a}}{x-w_{a}},\quad d(\gamma_{a}x)=-\frac{\rho_{a}}{(x-w_{a})^{2}}dx.
\end{align*}
With $\widetilde{P}_{a}(x)=\widetilde{p}_{a}(x)dx^N$ and $a\neq-b\in\I$ we find the infinite matrix product 
\begin{align*}
\widetilde{P}_{a}(x)\widetilde{R}_{ab}=& (-1)^{N}\sum_{m\ge 0}\frac{\rho_{a}^{m+N}}{(x-w_{a})^{m+2N}}\del^{(m)}\widetilde{p}_{b}(w_{-a})dx^N
\\
=&\sum_{m\ge 0}
\left(\frac{\rho_{a}}{x-w_{a}}\right)^{m}\del^{(m)}\widetilde{p}_{b}(w_{-a})d(\gamma_{a}x)^N  
=\widetilde{P}_{b}(\gamma_{a}x),
\end{align*}
provided $ |\gamma_{a}x-w_{-a}|<|w_{b}-w_{-a}|$. 
Assume that $x\in\D$ so that $x\notin \Delta_{a}$ and hence $\gamma_{a}x\in \Delta_{-a}$. 
Thus $\gamma_{a}x\notin \Delta_{b}$ and the inequality is satisfied using \eqref{eq:JordanIneq}. 
Similarly, for $a\neq-b$, $b\neq-c$ and $x\in\D$ we find  
\begin{align*}
\widetilde{P}_{a}(x)\widetilde{R}_{ab}\widetilde{R}_{bc}=\widetilde{P}_{b}(\gamma_{a}x)\widetilde{R}_{bc}=\widetilde{P}_{c}(\gamma_{b}\gamma_{a}x),
\end{align*}
since $\gamma_{a}x\in \Delta_{-a}$ implying $\gamma_{a}x\notin \Delta_{b}$ so that $\gamma_{b}\gamma_{a}x\in \Delta_{-b}$ and hence $\gamma_{b}\gamma_{a}x\notin \Delta_{c}$. 
For $x\in\D$ and $a_{1},\ldots,a_{k}\in\I $ with $a_{i}\neq -a_{i+1}$ for $i=1,\ldots,k-1$  then by induction
\begin{align}
\label{eq:PtildeRtildek}
\widetilde{P}_{a_{k}}(x)\widetilde{R}_{a_{k}a_{k-1}}\widetilde{R}_{a_{k-1}a_{k-2}}\ldots \widetilde{R}_{a_{2}a_{1}}=\widetilde{P}_{a_{1}}(\gamma_{a_{2}}\ldots \gamma_{a_{k}}x).
\end{align}
Assume that the $f_{\ell}(x)$ terms in  $\psi_{N}^{(0)}(x,y)$  are holomorphic for $|x-w_{a}|<r_{a}$ for  $a\in\I$ for some $r_{a}$. Recalling \eqref{eq:qdef} we find, much as above,  that
\begin{align*}
\widetilde{P}_{a}(x)Q_{a}(y)&=
\sum_{m\ge 0} \left(\frac{\rho_{a}}{x-w_{a}}\right)^{m}
\del^{(m,0)}\psi_{N}^{(0)}(w_{-a},y)
\left(\gamma_{a}x\right)^{N}dy^{1-N}
\\
&= \Psi_{N}^{(0)}(\gamma_{a}x,y),
\end{align*}
for $ |\gamma_{a}x-w_{-a}|< |y-w_{-a}|$ and $ r_{-a}$. The first inequality holds provided $x, y\in\D$ as above. Combining this result with  \eqref{eq:PtildeRtildek} we obtain
\begin{proposition}
Let $x,y\in\D$ and $a_{1},\ldots,a_{k}\in\I $ with $a_{i}\neq -a_{i+1}$. Then 
\begin{align}
\label{eq:PRkQ}
\widetilde{P}_{a_{k}}(x)\widetilde{R}_{a_{k}a_{k-1}}\widetilde{R}_{a_{k-1}a_{k-2}}\ldots \widetilde{R}_{a_{2}a_{1}}Q_{a_{1}}(y)=\Psi_{N}^{(0)} (\gamma x,y ),
\end{align}
for $ |\gamma x-w_{-a_{1}}|<  r_{-a_{1}}$ where $\gamma =\gamma_{a_{1}}\ldots \gamma_{a_{k}}$  is a reduced  word in  $\Gamma$ of length $k$. 
\end{proposition}
Assuming convergence (which depends on the choice of $f_{\ell}(x)$ terms),  we  may sum over all $a_{i}$ and use $\widetilde{R}_{a, -a}=0$ to find that for $k\ge 1$ 
\begin{align*}
\widetilde{P}(x)\widetilde{R}^{k-1}Q(y)
=\sum_{\gamma\in\Gamma_{k}}\Psi_{N}^{(0)}(\gamma x,y),
\end{align*}
where $\Gamma_{k}\subset \Gamma$ denotes the set of reduced words of length $k$. Then 
$\Psi_{N}(x,y)$ of \eqref{GenusgPsiDef} can be re-expressed using \eqref{eq:ImRinverse} as the following Poincar\'e sum
\begin{align}
\Psi_{N}(x,y) =\Psi_{N}^{(0)}(x,y)+\sum_{k\ge 1}\widetilde{P}(x)\widetilde{R}^{\,k-1}Q(y)
 =\sum_{\gamma\in\Gamma}\Psi_{N}^{(0)}(\gamma x,y).\label{PoincareSum}
\end{align}

We next consider a choice of $\Psi_{N}^{(0)}$ of \eqref{PsiDef} for which $\Psi_{N}$ is convergent  and is an example of what is referred to as a Green's function with Extended Meromorphicity  or  GEM form in~\cite{T1}.
For $N=1$ we let
\begin{align}\label{BersChoiceNeq1}
\Psi_{1}^{(0)}(x,y)=\left(\frac{1}{x-y}-\frac{1}{x}\right)dx.
\end{align} 
Then the  Poincar\'{e} sum~\eqref{PoincareSum} is convergent on $\D$ for $x\neq y$, $x\neq 0$ with~\cite{Bu,Bo} 
\begin{align}\label{eq:Psi1}
 \Psi_{1}(x,y)=\sum_{\gamma\in\Gamma}\Psi_{1}^{(0)}(\gamma x,y)=\omega_{y-0}(x),
 \end{align} 
where $\omega_{a-b}(x)$  is the differential of the third kind~\cite{Mu,Fa}. $ \Psi_{1}(x,y)$ is a 1-form in $x$ but is a weight zero differential quasi-form in $y$ where for   $a\in\Ip$ we find ~\cite{Bu}
\begin{align}\label{PsiDifferenceBurnside}
\Psi_{1}(x,y)-\Psi_{1}(x,\gamma_{a} y)=\nu_{a}(x),
\end{align}
where $\nu_{a}(x)$ is a holomorphic 1-form normalized by 
\begin{align}\label{NuNormalisation}
\frac{1}{\tpi}\oint_{\alpha_{a}}\nu_{b}  =  \delta_{ab},\quad a,b\in\Ip,
\end{align}
for  standard homology cycle  $\alpha_{a}\sim\calC_{-a}$. We also find
\begin{align}\label{eq:Psi1dy}
\del_{y}\Psi_{1}(x,y)dy=
\sum_{\gamma\in\Gamma}\frac{d(\gamma x)dy}{(\gamma x-y)^{2}}
=\omega(x,y),
\end{align}
where $\omega (x, y)$  is the symmetric  normalized bidifferential of the second kind~\cite{Mu,Fa} 
 
For all $N\ge 2$ we consider the Bers function~\cite{Be1}
\begin{align}\label{BersChoiceNge2}
\Psi_{N}^{(0)}(x,y)&=
\frac{1}{x-y}\prod_{j=0}^{2N-2}\frac{y-A_{j}}{x-A_{j}}dx^{N}dy^{1-N},
\end{align}
where $A_{0},\ldots,A_{2N-2}\in\Lambda(\Gamma) $ are distinct limit points for $\Gamma$.
Thus we find 
\begin{align*}
\psi_{N}^{(0)}(x,y)&=\frac{1}{x-y}-\sum_{i=0}^{2N-2}\frac{1}{x-A_{i}}L_{i}(y),
\end{align*}
where $L_{i}(y)=\prod_{j\neq i}\frac{y-A_{j}}{A_{i}-A_{j}}\in\Pi_{2N-2}(y)$, the space of degree $2N-2$ polynomials in $y$,
 determining the $f_{\ell}(x)$ coefficients of \eqref{PsiDef}. 
Then the Poincar\'{e} sum~\eqref{PoincareSum}  is holomorphic on $\D$ for $x\neq y$ ~\cite{Be1}.
$\Psi_{N} (x,y)$ is an $N$-form in $x$ and a $(1-N)$-form quasi-form in $y$ (generalizing \eqref{PsiDifferenceBurnside}) where for all $\gamma\in\Gamma$ ~\cite{Be1,McIT,T1}
\begin{align}
\label{eq:Psigamx}
\Psi_{N}(x,y)-\Psi_{N}( \gamma x,y) &=0,
\\
\label{eq:Psigamy}
\Psi_{N}(x,y)-\Psi_{N}( x,\gamma y) &=\sum_{r=1}^{d_{N}}\Phi_{r}(x){\Xi}_{r}[\gamma](y),
\end{align}
where $\{\Phi_{r}(x)\}_{r=1}^{d_{N}}$ is a basis\footnote{\label{foot:Dual basis}
	There is a difference  in notation here in comparison to that of~\cite{T1}. In  that paper the cohomology basis $\{\Xi_{r}\}_{r=1}^{d_{N}}$ is mapped isomorphically by the Bers map to  an $N$-form basis which is the Petersson dual of our $N$-form basis $\{\Phi_{r}(x)\}_{r=1}^{d_{N}}$.} of dimension  $d_{N}=(g-1)(2N-1)$ (by Riemann-Roch) of holomorphic $N$-forms 
and $\{\Xi_{r}\}_{r=1}^{d_{N}}$ is a cohomology basis
with Eichler $1$-cocycle elements of the form
${\Xi}_{r}[\gamma](y)=\xi_{r}[\gamma](y)dy^{1-N}$ for $\xi_{r}[\gamma](y)\in\Pi_{2N-2}(y)$ 
obeying 
\begin{align}
\label{eq:cocycle}
\Xi [\gamma\lambda](y)=\Xi [\gamma](\lambda y)+\Xi[\lambda](y),\quad \forall\;\gamma,\lambda\in\Gamma.
\end{align} 
A coboundary cocycle is one of the form $\Xi_{P}[\gamma](y)=P(\gamma y)-P(y)$ for  $\gamma\in \Gamma$ where  $P(y)=p(y)dy^{1-N}$  for $p(y)\in\Pi_{2N-2}(y)$.  $\{{\Xi}_{r}\}_{r=1}^{d_{N}} $ is then a basis for the  cohomology space of  cocycles modulo coboundaries. 
Lastly the $N$-forms are normalized~\cite{McIT,T1}
\begin{align}
\label{eq:Nformnorm}
\frac{1}{\tpi}
\sum_{a=1}^{g}\oint_{\calC_{a}}\Phi_{r}\,\Xi_{s}[\gamma_{a}]=
\delta_{rs},\quad r,s=1,\ldots,d_{N}.
\end{align}
We also define  a symmetric bidifferential of weight $(N,N)$ and  holomorphic for $x\neq y$:
 \begin{align}
 \label{eq:LambdaN}
 \Lambda_{N}(x,y):=
 \sum_{\gamma\in\Gamma}\left(\frac{d(\gamma x)dy}{(\gamma x-y)^{2}}\right)^{N}
=\left(
 \frac{1}{(x-y)^{2N}}+\textrm{regular terms}\right)
 dx^{N}dy^{N}.
 \end{align}
Then, since $\del_{y}^{(2N-1)}\psi_{N}^{(0)}(x,y)=(x-y)^{-2N}$, we may generalise \eqref{eq:Psi1dy} to find
\begin{align}
 \label{eq:LambdaNdlPsi}
dy^{2N-1}\del_{y}^{(2N-1)}\Psi_{N}(x,y)=\Lambda_{N}(x,y).
\end{align}

\subsection{The holomorphicity of $ \Theta_{a}(x;\ell)$} 
For $N\ge 2$ with the choice \eqref{BersChoiceNge2}, then  $\Psi_{N}(x,y)$ is holomorphic for $x,y\in\D$ with $x\neq y$. 
For $a\in\I $ we consider the Laurent expansion 
$\psi_{N}(x,y)=\sum_{n\in\Z}c_{a}(x;n)y_{a}^n  $ for $y_{a}=y-w_{a}$.
\eqref{eq:Psigamx} implies that $c_{a}(x;n)dx^{N}$ is a holomorphic $N$-form. In particular, for $\ell=0,\ldots,2N-2$ and  $a\in\I $ and using \eqref{eq:pdef}, \eqref{eq:qdef}, \eqref{eq:Rdef} and \eqref{GenusgPsiDef} we  obtain  
\begin{align*}
c_{a}(x;\ell)&=\frac{1}{\tpi}\oint_{\calC_{a}}\psi_{N}(x,y)y_{a}^{-\ell-1}dy 
=\chi_{a}(x;\ell),
\end{align*}
 of \eqref{eq:chiadef}. Therefore $\chi_{a}(x;\ell)dx^{N}$ and  $\Theta_{a}(x;\ell)=\theta_{a}(x;\ell)dx^{N}$ of  \eqref{eq:thetadef}
are  holomorphic $N$-forms   for all $N\ge 2$ for the Bers choice.

Consider $ \gamma_{a}y=w_{-a}+\rho_{a}/y_{a}$ for $a\in\Ip$.
From \eqref{PsiDifferenceBurnside} and \eqref{eq:Psigamy},  $\Psi_{N}(x,y)-\Psi_{N}(x,\gamma_{a}y)$ is a degree $2N-2$ polynomial in $y$ with coefficients given by holomorphic $N$-forms which may computed by identifying the following terms in the Laurent expansions 
\begin{align*}
\Psi_{N}(x,y)&=\ldots +\sum_{\ell=0}^{2N-2} \chi_{a}(x,\ell) y_{a}^\ell  dx^{N}dy_{a}^{1-N}+\ldots
\\
\Psi_{N}(x,\gamma_{a}y)&=\ldots +\sum_{\ell=0}^{2N-2} \chi_{-a}(x,\ell) y_{-a}^\ell dx^{N}dy_{-a}^{1-N}+\ldots
\end{align*}
where $y_{-a}=\rho_{a}/y_{a}$. Hence we find using \eqref{eq:thetadef} that 
\begin{align}
\Psi_{N}(x,y)-\Psi_{N}(x,\gamma_{a}y)
&=\sum_{\ell=0}^{2N-2}\left(\chi_{a}(x,\ell) y_{a}^\ell 
- \chi_{-a}(x,\ell) \left(\frac{\rho_{a}}{y_{a}}\right)^\ell \left(\frac{-\rho_{a}}{y_{a}^2}\right)^{1-N} \right)dx^{N}dy_{a}^{1-N} 
\notag
\\
&=\sum_{\ell=0}^{2N-2} \Theta_{a}(x,\ell)  y_{a}^\ell dy_{a}^{1-N}.
\label{eq:ThetaPsi}
\end{align}
Thus using \eqref{PsiDifferenceBurnside} we conclude that for $N=1$
\begin{align}
\label{eq:ThetaN1}
\Theta_{a}(x;0)=\nu_{a}(x),
\end{align}
for normalized 1-forms.
For $N\ge 2$, 	\eqref{eq:ThetaPsi}  can be compared to \eqref{eq:Psigamy} for $\gamma=\gamma_{a}$ with canonical cocycle basis \eqref{eq:xinorm}. Furthermore, since $\Gamma$ is generated by $\gamma_{a}$ then \eqref{eq:Psigamy} implies that $\{\Theta_{a}(x;\ell)\}$ for $\ell=0,\ldots,2N-2$ and $a\in\Ip $  spans the space of holomorphic $N$-forms. Lastly, we note \eqref{eq:ThetaPsi} implies that $ \Theta_{a}(x,\ell)$ can be written as a Poincar\'{e} series (and therefore so can every holomorphic $N$-form).
 
The following summary theorem  describes the convergence and geometrical meaning of the universal coefficients $\{\Theta_{a}(x;\ell)\}$ for $a\in\Ip$ and $\ell=0,\ldots,2N-2$ and $\Psi_{N}(x,y)$ 
  in the Zhu recursion formulas of Theorem~\ref{theor:ZhuGenusg} for $g\ge 2$.  
\begin{theorem}\label{theor:ZhuConvergent}
	For  $N=1$ we may choose $\Psi_{1}^{(0)}$ of \eqref{BersChoiceNeq1} so that  $\Psi_{1}(x,y)=\omega^{(g)}_{y-0}(x)$ and  $\Theta_{a}(x;0)=\nu_{a}^{(g)}(x)$. 
	For $N\ge 2$ we may make the Bers choice for $\Psi_{N}^{(0)}$ in \eqref{BersChoiceNge2} so that
$\Psi_{N}(x,y)$ is  holomorphic for $x\neq y$ and is an $N$-form in $x$ and a $(1-N)$ quasidifferential  in $y$ and where  $\{\Theta_{a}(x;\ell)\}$  spans the space of holomorphic $N$-forms.
\end{theorem}
Combining Theorem~\ref{theor:ZhuGenusg}, Corollary~\ref{cor:GenGenusgZhu} and  \ref{theor:ZhuConvergent} we find that  $\F_{V}^{(g)}(u,x;\bm{v,y})$ is a finite linear combination, with formal coefficients, of meromorphic functions in $x$ with poles at $x=y_{k}\in\D$.
Furthermore we may identify the $V$-dependent formal coefficients $O_{a}(u;\bm{v,y})$ in \ref{eq:ZhuGenusg} as follows
\begin{lemma}\label{lem:Oacont}
For $a\in\Ip$, $0\le \ell\le 2N-2$  and $y_{k}\in\D$ we have 
\[
O_{a}(u;\bm{v,y};\ell)=\frac{1}{\tpi} \oint_{\mathcal{C}_a}\F_{V}^{(g)}(u,x;\bm{v,y})(x-w_{a})^{\ell}\,dx^{1-N}.
\]	
\end{lemma}
\begin{proof}
 Let $x=x_{a}+w_{a}$ and using locality and associativity \eqref{eq:Z0assoc} (up to a common multiplier) we find
\begin{align*}
\F_{V}^{(g)}(u,x;\bm{v,y})&=\sum_{\bm{b}_{+}}\F^{(0)}(u,x_{a}+w_{a};\bm{v,y};\bm{b,w})
\\
&=\sum_{\bm{b}_{+}}\F^{(0)}( \ldots ;Y(u,x_{a})b_{a},w_{a};\ldots)
\\
&=\sum_{\bm{b}_{+}}\sum_{k\in\Z}\F^{(0)}( \ldots ;u(k)b_{a},w_{a};\ldots)x_{a}^{-k-1}.
\end{align*}
The result follows from the meromorphicity of $\F_{V}^{(g)}(u,x;\bm{v,y})$ in $x$.
\end{proof}
\subsection{General and canonical GEM forms}\label{sect:normalized basis}
For $N\ge 2$ we  may define a canonical  basis of $g(2N-1)$  cocycles associated with the Schottky sewing scheme as follows~\cite{T1}. For $a,b\in \Ip$ and $k=0,\ldots,2N-2$ and for $\Gamma$ generator $\gamma_{b}$  we define
\begin{align}
\label{eq:xinorm}
\Xi_{ak}[\gamma_{b}](y):=\delta_{ab}(y-w_{a})^{k}\, dy^{1-N}.
\end{align}
\eqref{eq:cocycle} determines the cocycle for all $\gamma\in\Gamma$.
There are $2N-1$  relations modulo coboundaries on $\{\Xi_{ak}\}$. We may choose a cohomology basis  $\{\Xi_{ak}\}_{\J}$  indexed by $ (a,k)\in \J$, a set  of  $d_{N}$ appropriate  $(a,k)$ values with $1\le a \le g$ and $0\le k\le 2N-2$. Let $\{\Phi_{ak}\}_{\J}$ be the corresponding basis of holomorphic $N$-forms with normalization \eqref{eq:Nformnorm}. We may thus define a more general GEM form \cite{T1}
\begin{align}\label{eq:GEMform}
\Psi_{N}(x,y):=\Psi_{N}^{\Bers}(x,y)+	\sum_{(a,k)\in \J}\Phi_{ak}(x)P_{ak}(y),
\end{align}
where $\Psi_{N}^{\Bers}(x,y)$ is the Bers choice  and $P_{ak}(y)=p_{ak}(y)dy^{1-N}$ for  $p_{ak}(y)\in\Pi_{2N-2}(y)$. 
 $\Psi_{N}$  is holomorphic on $\D$ for $x\neq y$ and obeys \eqref{eq:Psigamx}
and \eqref{eq:Psigamy} with 
\[
\Xi_{ak}[\gamma]=\Xi_{ak}^{\Bers}[\gamma]+\Xi_{P_{ak}}[\gamma],
\] 
for coboundary cocycle $\Xi_{P_{ak}}[\gamma](y)=P_{ak}(\gamma y)-P_{ak}( y)$. Furthermore,  Zhu recursion of Theorems~\ref{theor:ZhuGenusg} and \ref{theor:ZhuGenusgModules} hold for any choice of GEM form \eqref{eq:GEMform} with
$\del^{(0,j)}\Psi_{N}(x,y)=\del^{(0,j)}\Psi_{N}^{\Bers}(x,y)$ for all $j\ge 2N-1$. In particular  \eqref{eq:LambdaNdlPsi} holds for all GEM forms. 
\medskip

We may  choose  $P_{ak}(y)$ in order to construct a canonical GEM form\footnote{See footnote~\ref{foot:Dual basis}} as  follows~\cite{T1} 
\begin{proposition}\label{prop:Can_norm}
	For $N\ge 2$ and cohomology basis $\{\Xi_{ak}\}_{\J}$  with normalized $N$-form basis $\{\Phi_{ak}\}_{\J}$, there exists a canonical GEM form ${\Psi}^{\Can}_{N}$ such that for $(a,k),(b,l)\in\J$ 
	\begin{align}\label{eq:CanPhi_norm}
	&\frac{1}{\tpi} \oint_{\mathcal{C}_a}\Phi_{bl}(x)(x-w_{a})^k\,dx^{1-N}
	=
 \delta_{ab}\delta_{kl},
	\\
	& {\Psi}^{\Can}_{N}(x,y)- {\Psi}^{\Can}_{N}( x,\gamma_{a} y) 
	=
	\sum_{(a,k)\in \J}\Phi_{ak}(x)(y-w_{a})^k\,dy ^{1-N},
	\label{eq:CanPsiQPeriod}
	\\
	&\oint_{\calC_{a}} {\Psi}^{\Can}_{N}(x,y)(x-w_{a})^k\,dx ^{1-N}
	=0,\quad \mbox{for }y\in\D.
	\label{eq:CanPsi_norm}
	\end{align} 
\end{proposition}
Proposition~\ref{prop:Can_norm} is a natural generalisation of the properties of $ \omega _{y-0}(x)$ of \eqref{eq:Psi1} and \eqref{NuNormalisation} where, for $N=1$, we have $\J=\{(a,0)|a\in\I \}$, 
${\Psi}^{\Can}_{1}(x,y)=\omega^{(g)}_{y-0}(x)$ and $\Phi_{a0}=\nu_{a}$.
Since $\Psi_{N}^{\Bers}(x,y)$ and $\Phi_{ak}(x)$ have Poincar\'e sums,  so also does the canonical GEM form  for an appropriate  choice of the coefficient functions $f_{\ell}^{\Can}(x)$ in $\psi^{(0)}_{N}(x,y)$. 
We therefore obtain the following refinement of the genus $g$ Zhu recursion Theorem~\ref{theor:ZhuGenusg} (where a spanning set of $N$-forms is replaced by a basis):
\begin{theorem}[Normalized Quasiprimary Genus $g$ Zhu Recursion]\label{theor:ZhuGenusgCan}
	For a canonical cohomology  basis $\{\Xi_{ak}\}_{\J}$ with corresponding holomorphic $N$-form basis $\{\Phi_{ak}\}_{\J}$ and canonical GEM form ${\Psi}^{\Can}_{N}$, the genus $g$ $n$-point formal differential for a quasiprimary vector $u$ of weight $\wt(u)=N$ satisfies the recursive identity
	\begin{align}\label{eq:ZhuGenusgCan}
	\F_{V}^{(g)}(u,x;\bm{v,y})&=\sum_{(a,k)\in\J} \Phi_{ak}(x)O_{a}(u;\bm{v,y};k)
	\notag
	\\
	&\quad+\sum_{r=1}^{n}\sum_{j\ge 0}\del^{(0,j)}\Psi^{\Can}_{N}(x,y_{r})dy_{r}^{j}\,\F_{V}^{(g)}(\ldots;u(j)v_{r},y_{r};\ldots).
	\end{align}
	\end{theorem} 
\begin{proof}
With the canonical choice $\Psi^{\Can}_{N}$ the RHS of \eqref{eq:ZhuGenusg} has leading  contribution
\begin{align*}
\sum_{b=1}^{g} \Theta^{\Can}_{b}(x)O_{b}(u;\bm{v,y})
=\sum_{(a,k)\in\J} \Phi_{ak}(x)A_{ak}(u;\bm{v,y}),
\end{align*}
with $\Theta^{\Can}_{a}(x;\ell)$ determined by \eqref{eq:ThetaPsi} for $\Psi^{\Can}_{N}$ and where $A_{ak}(u;\bm{v,y}) $ is a linear combination of  $O_{b}(u;\bm{v,y};\ell)$ terms. 
But Lemma~\ref{lem:Oacont} and Proposition~\ref{prop:Can_norm} imply that for $(a,k)\in\J$ and with $x_{a}=x-w_{a}$
\begin{align*}
O_{a}(u;\bm{v,y};k)
=&\sum_{(b,l)\in\J} \frac{1}{\tpi} \oint_{\mathcal{C}_a}\Phi_{bl}(x)x_{a}^k\,dx^{1-N}A_{bl}(u;\bm{v,y})
\\
&+\sum_{r=1}^{n}\sum_{j\ge 0}
\frac{1}{\tpi} \oint_{\mathcal{C}_a}\del^{(0,j)}\Psi^{\Can}_{N}(x,y_{r})x_{a}^k\,dx^{1-N}\,\F_{V}^{(g)}(\ldots;u(j)v_{r},y_{r};\ldots)dy_{r}^{j}.
\\
=&A_{ak}(u;\bm{v,y})+0.
\end{align*}
\end{proof}

\section{Some Applications}
\label{sec:applications}
In this section we consider some applications of genus $g$ Zhu recursion. We analyse all $n$-point correlation functions for the rank 1 Heisenberg VOA $M$ using $\Psi_{1}(x,y)=\omega_{y-0}(x)$. 
We also consider Virasoro $n$-point functions for general VOAs based on any choice of $\Psi_{2}(x,y)$. We determine various  genus $g$ conformal Ward identities which are determined in terms of variations with respect to the Schottky parameters $\Cg$ and puncture coordinates on a punctured Riemann surface. 
Thus we explicitly realize  the relationship between Virasoro correlation functions and variations in moduli space exploited in conformal field theory  e.g.~\cite{EO}.
Our analysis leads to several partial differential equations giving the variation with respect to (punctured) Riemann surface moduli of the Heisenberg VOA partition function $Z_{M}^{(g)}$, the bidifferential    $\omega(x,y)$, the projective connection $s(x)$,  
the holomorphic 1-forms   $\nu_{a}$  and the period matrix $\Omega_{ab}$. The deeper relationship between $\Psi_{2}(x,y)$ and the moduli variations  is described in~\cite{T1}. Finally, we compute the genus $g$ partition function for an even  lattice VOA  in terms of the Siegel lattice theta series and $Z_{M}^{(g)}$.
\subsection{The Heisenberg VOA}
%
The rank one Heisenberg VOA $M$  is generated by a weight one vector $h$ with
\begin{align*}
[h(m),h(n)]=m\delta_{m,-n}.
\end{align*}
Genus $g$ Zhu recursion implies (which generalises genus two results in~\cite{MT2,GT})
\begin{proposition}
	The $n$-point differential $\F_{M}^{(g)}(\bm{h,x})=0$  for odd $n$ and for even $n$ is given by
\begin{align*}
\F_{M}^{(g)}(\bm{h,x})= \Sym_{n}\omega(\bm{x}) \; Z_{M}^{(g)},
\end{align*}
where $\bm{h,x}:=h,x_{1};\ldots;h,x_{n}$ and $\Sym_{n}\omega(\bm{x})  $ denotes the symmetric product
\begin{align*}
\Sym_{n}\omega(\bm{x})  =\sum_{\sigma}\prod_{(r,s)}\omega (x_{r},x_{s}),
\end{align*}
where we sum over all fixed-point-free involutions $\sigma=\ldots (rs) \ldots$ of the labels $\{1,2,\ldots,n\}$.
\end{proposition}
\begin{proof}
	Since $h(0)v=0$ for every $v\in M$ it follows that $O(h;\bm{h,x})=0$ so that, in particular, $\F_{M}^{(g)}(h,x)=0$. Theorem~\ref{theor:ZhuGenusg}, \eqref{eq:Psi1dy} and the fact that $h(j)h=\delta_{j1}\vac$ for $j\ge 0$  imply
\begin{align*}
\F_{M}^{(g)}(\bm{h,x})=&\sum_{k=2}^{n}\sum_{j\ge 0}\del^{(0,j)}\Psi_{1}(x_{1},x_{k})\, dx_{k}^{j}\F_{M}^{(g)}(h,x_{2};\ldots;h(j)h,x_{k};\ldots;h,x_{n})\\
=&\sum_{k=2}^{n}\omega(x_{1},x_{k})\F_{M}^{(g)}(h,x_{2};\ldots;\widehat{h,x_{k}};\ldots;h,x_{n}),
\end{align*}
where $h$ at $x_{k}$ is omitted. The proof follows by induction.
\end{proof}
$\F_{M}^{(g)}(\bm{h,x})$ is a generating function for all genus $g$ $n$-point differentials of the Heisenberg VOA much as in~\cite{MT3} (Prop. 3.8). Thus the Virasoro 1-point function can be found as follows. 
Using VOA associativity \eqref{eq:Z0assoc} we firstly note
\begin{align}
\F_{M}^{(g)}(h,x;h,y)&= \sum_{m\in\Z}\F_{M}^{(g)}(h(m)h,y)(x-y)^{-m-1}
\notag
\\
&=\frac{Z_{M}^{(g)}dxdy}{(x-y)^{2}}+\sum_{m\le - 1}\F_{M}^{(g)}(h(m)h,y)(x-y)^{-m-1}.
\label{eq:FMhh}
\end{align}
But  the Virasoro vector  $\omega=\frac{1}{2}h(-1)h$ so that 
\begin{align}
\F_{M}^{(g)}(\omega,x)=&
\lim_{y\rightarrow x}\frac{1}{2}\left(\F_{M}^{(g)}(h,x;h,y)-\frac{Z_{M}^{(g)}dxdy}{(x-y)^{2}}\right)
=\frac{1}{12}s(x)Z_{M}^{(g)},
\label{eq:VirasoroProjConn}
\end{align}
for genus $g$ projective connection  $s(x)$ defined by
\[
s(x):=6\lim_{y\rightarrow x}\left( \omega(x,y)-\frac{dxdy}{(x-y)^2}\right)
=\sum_{\gamma\in\Gamma,\gamma\neq \Id}\frac{d(\gamma x)dx}{(\gamma x-x)^{2}}.
\]
\subsection{Virasoro $1$-point differentials and variations of moduli}
Consider the genus $g$ Virasoro 1-point differential
\begin{align*}
\F_{V}^{(g)}(\omega,x)=\sum_{a=1}^{g} \sum_{\ell=0}^{2}\Theta_{a}(x,\ell)O_{a}(\omega;\ell),
\end{align*}
where $\{ \Theta_{a}(x,\ell) \}$ spans the $3g-3$ dimensional space of holomorphic $2$-forms. In order to analyse the $O_{a}(\omega;\ell)$ terms we define  differential operators
$\del_{a,\ell}$  for $a\in\Ip$ and $\ell=0,1,2$ as follows:
\begin{align}
\label{eq:delael}
\del_{a,0}:=\del_{w_{a}},\quad
\del_{a,1}:=\rho_{a} \del_{\rho_{a}},\quad
\del_{a,2}:=\rho_{a} \del_{w_{-a}}.
\end{align}
These form a canonical basis for the tangent space $T(\Cg)$~\cite{T1}. 
\begin{lemma} For $a\in\Ip$ and $\ell=0,1,2$ we find
	\begin{align}\label{eq:Oaomega}
	O_{a}(\omega;\ell)=\del_{a,\ell}Z_{V}^{(g)}.
	\end{align}
\end{lemma}
\begin{proof}
For $\ell=0$ with $\omega(0)=L(-1)$ this follows since
\begin{align*}
O_{a}(\omega;0)&=\sum_{\bm{b_{+}}}Z^{(0)}(\ldots;L(-1)b_{a},w_{a};\ldots)=\del_{w_{a}}Z_{V}^{(g)}.
\end{align*}
Similarly  for $\ell=1$ with $\omega(1)=L(0)$ we have
\begin{align*}
O_{a}(\omega;1)=\sum_{\bm{b_{+}}}Z^{(0)}(\ldots;L(0)b_{a},w_{a};\ldots)=\sum_{\bm{b_{+}}}\wt(b_{a})Z^{(0)}(\bm{b,w}).
\end{align*}
But recalling \eqref{eq:bbar}, it follows that $\rho_{a}\del_{\rho_{a}}b_{-a}=\wt(b_{a})b_{-a}$ so that
\begin{align*}
O_{a}(\omega;1)=\rho_{a}\del_{\rho_{a}}Z_{V}^{(g)}.
\end{align*}
Lastly, for $\ell=2$ with $\omega(2)=L(1)$ and  using the adjoint relations \eqref{eq:adjointrel} and \eqref{RhoAdjoint} we find
\begin{align*}
O_{a}(\omega;2)&=\sum_{\bm{b_{+}}}Z^{(0)}(\ldots;L(1)b_{a},w_{a};\ldots)
\\
&=\sum_{\bm{b_{+}}}\rho_{a}Z^{(0)}(\ldots;L(-1)b_{-a},w_{-a};\ldots)
\\
&
=\sum_{\bm{b_{+}}}\rho_{a}\del_{w_{-a}}Z^{(0)}(\bm{b,w})
=\rho_{a}\del_{w_{-a}}Z_{V}^{(g)}.
\end{align*}
\end{proof}
Define the  differential operator ~\cite{GT, T1}
\begin{align}\label{eq:NablaDef}
\delCg(x):=\sum_{a=1}^{g}\sum_{\ell=0}^{2}\Theta_{a}(x,\ell)\del_{a,\ell}.
\end{align}
\begin{proposition}\label{prop:Virasoro1pt}
The Virasoro $1$-point differential is given by 
\begin{align}\label{eq:VirasoroNabla}
\F_{V}^{(g)}(\omega,x)=\delCg(x)Z_{V}^{(g)}.
\end{align}
\end{proposition}
\begin{remark}
	\eqref{eq:VirasoroNabla} is a  generalization of the    formula  
	$Z^{\Zhu}_{V}({\omega},x;q)=q\del_{q} Z^{\Zhu}_{V}(q)$ at genus one~\cite{Z}
	and a similar genus two result  in~\cite{GT}.  
\end{remark}
The geometrical meaning of $\delCg(x)$ in terms of variations in the Schottky parameter space $\Cg$ and the dependence of $\delCg(x)$ on the choice of GEM form $\Psi_{2}$ is   discussed in ~\cite{T1}. In particular,  if we consider a new GEM form $\widecheck{\Psi}_{2}(x,y)=\Psi_{2}(x,y)-\Phi(x)p(y)dy^{-1}$ for a holomorphic 2-form $\Phi(x) $ and polynomial $p(y)\in \Pi_{2}(y)$, then one finds 
\begin{align*}
\delCgcheck(x)=\delCg(x)+\Phi(x)\mathcal{L}_{p}^{\Cg},
\end{align*}
for $\slLie_{2}(\C)$ generator $\mathcal{L}_{p}^{\Cg}$  of \eqref{eq:calLdef}. Thus   Proposition~\ref{prop:ZSL2C} implies \eqref{eq:VirasoroNabla} holds for all choices of  $\Psi_{2}$. In general, $\SL_{2}(\C)$ invariance of $\Cg$ of \eqref{eq:Mobwrhoa} implies that $\delCg(x)$ defines a natural differential operator $\delMg$ on the $\Sg$ moduli space $\Mg$~\cite{T1}. In particular, for the canonical GEM form of Proposition~\ref{prop:Can_norm}  we have
\begin{align}
\label{eq:delSchottky}
\delMg(x)=\sum_{(a,\ell)\in\J}\Phi_{a\ell}(x)\partial_{a,\ell},
\end{align} 
with canonical holomorphic 2-form basis\footnote{See footnote~\ref{foot:Dual basis}} $\{\Phi_{a\ell}\} _{(a,\ell)\in\J}$ where $\J$ is a set of $3g-3$ distinct $(a,\ell)$ values as described in \S\ref{sect:normalized basis}.  In fact, for any moduli coordinates $\{m_{r}\} _{r=1}^{3g-3}$ on $\Mg$ there exists a corresponding  2-form basis  $\{\Phi_{r}\} _{r=1}^{3g-3}$ such that ~\cite{T1,O}
\begin{align}
\label{eq:nablaMg}
\delMg(x)=\sum_{r=1}^{3g-3}\Phi_{r}(x)\partial_{m_{r}}.
\end{align}
In summary Proposition~\ref{prop:Virasoro1pt} and \ref{prop:nptSL2CProp} imply:
\begin{proposition}\label{prop:nablaMgZ}
The Virasoro $1$-point differential is given by 
\begin{align}\label{eq:delMg}
\F_{V}^{(g)}(\omega,x)=\delMg(x)Z_{V}^{(g)}.
\end{align}
\end{proposition}

\medskip

For  the rank one Heisenberg VOA $M$, using~\eqref{eq:VirasoroProjConn} and~\eqref{eq:VirasoroNabla} we obtain  a generalization of a genus two result~\cite{GT}
\begin{proposition} The genus $g$ partition function for the rank one Heisenberg VOA obeys the linear partial differential equation
\begin{align}\label{prop:BosonDE}
\left(\delMg(x)-\frac{1}{12}s(x)\right)Z_{M}^{(g)}=0.
\end{align}
\end{proposition}
\begin{remark}
We note  that for $\widetilde{R}$ of \eqref{eq:Rtilde} with $N=1$ then 
	$Z_{M}^{(g)}=\det(1-\widetilde{R})^{-\frac{1}{2}}$~\cite{TZ}. 
	This is discussed further in~\cite{T2}.
\end{remark}

\subsection{Genus $g$ Ward identities}
We define a derivative operator on the parameter space $\Cgn:=\Cg\times(\Sg)^{n}$ for a Riemann surface $\Sgn$ with $n$  punctures  $y_{1},\ldots,y_{n}$~\cite{T1,O}
\begin{align}
\label{eq:nablaCgn}
\delCgn(x):=\delCg(x)+\sum_{k=1}^{n}\Psi_{2}(x,y_{k})dy_{k}\,\partial_{y_{k}}.
\end{align}
The quasiperiodicity property \eqref{eq:Psigamy} implies that $\delCgn(x)$ is a vector field on $\Cgn$~\cite{T1,O}. Using Theorems \ref{theor:ZhuGenusg} and \ref{theor:ZhuConvergent}, we can derive genus $g$ Ward identities of physics e.g.~\cite{EO}. 
\begin{proposition}\label{prop:nptWard}
For primary vectors $v_{k}$ of weight $\wt(v_{k})$ for $k=1,\ldots,n$, we find 
\begin{align}\label{eq:Ward1}
\F_{V}^{(g)}(\omega,x;\bm{v,y})=\left(
\delCgn(x)
+\sum_{k=1}^{n}\wt(v_{k})dy_{k}\,\Psi_{2}^{(0,1)}(x,y_{k})
\right)\F_{V}^{(g)}(\bm{v,y}).
\end{align}
\end{proposition}
We also have a Ward identity for Virasoro vector $\omega$ insertions at  $x,y_{1},y_{2},\ldots,y_{n} $.
\begin{proposition}\label{prop:omWard}
The genus $g$ Virasoro $n$-point differential obeys the identity
\begin{align}
\F_{V}^{(g)}(\omega,x;\bm{\omega,y})&
=\left(\delCgn(x)
+2\sum_{k=1}^{n}dy_{k}\,\Psi_{2}^{(0,1)}(x,y_{k})
\right)\F_{V}^{(g)}(\bm{\omega,y})
\notag
\\
&\quad+\frac{C}{2}\sum_{k=1}^{n}\Lambda_{2}(x,y_{k})\F_{V}^{(g)}(\ldots;\widehat{\omega,y_{k}};\ldots),\label{eq:Ward2}
\end{align}
for $\Lambda_{N}$ of \eqref{eq:LambdaN} and central charge $C$  and where $\omega$ at $y_{k}$ is omitted.
\end{proposition}
\begin{remark}\label{rem:WardModule}
Propositions~\ref{prop:nptWard} and \ref{prop:omWard}  are generalizations of genus two results in~\cite{GT}. 
We also note that similar expressions hold for $U$-module $n$-point functions for a subVOA $U$ of $V$ of \eqref{eq:Z_Walpha} by Theorem~\ref{theor:ZhuGenusgModules}.
\end{remark}
The Ward identities \eqref{eq:Ward1} and  \eqref{eq:Ward2} are independent of the choice of GEM function as follows. 
Consider a new GEM form $\widecheck{\Psi}_{2}(x,y)=\Psi_{2}(x,y)-\Phi(x)p(y)dy^{-1}$ for a holomorphic 2-form $\Phi(x) $ and polynomial $p(y)\in \Pi_{2}(y)$. Then one finds~\cite{T1}
\begin{align*}
\delCgncheck(x)=\delCgn(x)+\Phi(x)\mathcal{L}_{p}^{\Cgn},
\end{align*}
for $\mathcal{L}_{p}^{\Cgn}$  of \eqref{eq:calLdef}. Thus with the new GEM form  $\widecheck{\Psi}_{2}(x,y)$ equation  \eqref{eq:Ward1} implies
\begin{align*}
\widecheck{\F}_{V}^{(g)}(\omega,x;\bm{v,y})-{\F}_{V}^{(g)}(\omega,x;\bm{v,y})
=
\Phi(x)\left(\mathcal{L}_{p}^{\Cgn}-\sum_{k=1}^{n}\wt(v_{k})\del_{{k}} p(y_{k})\right)
\F_{V}^{(g)}(\bm{v,y})=0,
\end{align*}
using $\SL_{2}(\C)$ invariance  \eqref{eq:Lpnpt}. A similar argument applies to  \eqref{eq:Ward2} recalling that $\Lambda_{2}$ is independent of the choice of GEM form using \eqref{eq:LambdaNdlPsi}.
As before, $\SL_{2}(\C)$ invariance of $\Cgn$  implies that $\delCgn(x)$ determines a  differential operator $\delMgn$ on $\Sgn$ moduli space $\Mgn$ where, in terms of a canonical GEM form, we have~\cite{T1,O} 
\begin{align}\label{eq:delMgn}
\delMgn(x)=\delMg(x)+\sum_{k=1}^{n}dy_{k}\Psi_{2}^{{\Can}}(x,y_{k})\partial_{y_{k}}.
\end{align}
\begin{remark}\label{rem:Mgn}
Much as for Proposition~\ref{prop:nablaMgZ}, $\SL_{2}(\C)$ invariance of \eqref{eq:Ward1} and \eqref{eq:Ward2} imply we may choose a canonical GEM form $\Psi_{2}^{{\Can}}(x,y_{k})$ and replace $\delCgn$ by $\delMgn$ in Propositions~\ref{prop:nptWard} and \ref{prop:omWard}.
\end{remark}

\subsection{Some genus $g$ linear partial differential equations}
We now exploit some Heisenberg and Virasoro $n$-point correlation functions to obtain linear partial  differential equations for the projective connection $s(x)$, the bidifferential form $\omega(x,y)$, the holomorphic 1-form $\nu_{a}(x)$ and the period matrix $\Omega_{ab}$ given by
\begin{align}\label{eq:Omdef}
\tpi\Omega_{ab}:=\oint_{\beta_{a}}\nu_{b},  \qquad a,b\in\Ip.
\end{align}  
\begin{proposition}\label{prop:delsomnuOm}
$s(x)$, $\omega(x,y)$, $\nu_{a}(x)$ and $\Omega_{ab}$ satisfy the  differential equations:
\begin{align}
\label{omegaDE}
&\left(\delCgone(x) +2 \Psi_{2}^{(0,1)}(x,y)dy\right)s(y)
=6 \left(\omega(x,y)^{2}-\Lambda_{2}(x,y)\right),  
\\[3pt]
&\left(\delCgtwo(x) + \Psi_{2}^{(0,1)}(x,y)dy+  \Psi_{2}^{(0,1)}(x,z)dz\right)\omega(y,z)=\omega(x,y)\omega(x,z),\label{NablaOmega}
\\[3pt]
&\left(\delCgone(x) +\Psi_{2}^{(0,1)}(x,y)dy\right)\nu_{a}(y)=\omega(x,y)\nu_{a}(x),\quad  a\in\Ip, \label{NablaNuDifferentialEqn}
\\[3pt]
&\; \tpi\delCg(x)\Omega_{ab}=\nu_{a}(x)\nu_{b}(x),\quad  a,b\in\Ip,\label{eq:Rauch}
\end{align}
where $\delCgone(x)=\delCg(x)+dy\,\Psi_{2}(x,y)\partial_{y}$ and $\delCgtwo(x)=\delCgone(x) +dz\,\Psi_{2}(x,z)\partial_{z}$.
\end{proposition}
\begin{remark}  These results are generalisations of genus two formulas~\cite{GT}.
	Following Remark~\ref{rem:Mgn} we may again choose a canonical GEM form $\Psi_{2}^{{\Can}}(x,y_{k})$ and replace $\delCgn$ by $\delMgn$ in Proposition~\ref{prop:delsomnuOm}. In particular, equation~\eqref{eq:Rauch} is equivalent to Rauch's formula~\cite{R}.
The operators $\delMg(x)$ and $\delMgn(x)$ are also discussed  in \cite{O} but the existence and a construction of the GEM form $\Psi_{2}$ was not shown. Some similar partial differential equations for Riemann surface structures  (including the prime form) appear in \cite{O}.
\end{remark}
\begin{proof}
Consider the rank one Heisenberg VOA $M$. Using a similar approach to that leading to \eqref{eq:FMhh} and \eqref{eq:VirasoroProjConn} we find
\begin{align}\label{omegaDE1}
&\F_{M}^{(g)}(\omega,x;\omega,y)=\left(\frac{1}{144}s(x)s(y)+\frac{1}{2}\omega(x,y)^{2}\right)Z_{M}^{(g)}.
\end{align}
But using Proposition~\ref{prop:omWard} (with $C=1$) 
 and \eqref{eq:VirasoroProjConn}  we have
\begin{align*}
\F_{M}^{(g)}(\omega,x;\omega,y) 
&=\left(\delCgone(x) +2\Psi_{2}^{(0,1)}(x,y)dy\right)\left(\frac{1}{12}s(y)Z_{M}^{(g)}\right)+\frac{1}{2}\Lambda_{2}(x,y)Z_{M}^{(g)}.
\end{align*}
This together with \eqref{prop:BosonDE} implies \eqref{omegaDE}. 

Similarly, Proposition~\ref{prop:nptSL2CProp} and \ref{prop:BosonDE} 
imply
\begin{align*}
&\F_{M}^{(g)}(\omega,x;h,y;h,z)=\left(\delCgtwo+\Psi_{2}^{(0,1)}(x,y)dy+\Psi_{2}^{(0,1)}(x,z)dz
+\frac{1}{12}s(x)\right)\omega(y,z)Z_{M}^{(g)}.
\end{align*}
Alternatively,  
\begin{align*}
\F_{M}^{(g)}(\omega,x;h,y;h,z)&=\lim_{x_{1},x_{2}\rightarrow x}\left(\F_{M}^{(g)}(h,x_{1};h,x_{2};h,y;h,z)-\frac{\F_{M}^{(g)}}{(x_{1}-x_{2})^{2}}\right)
\\
&=\left(\frac{1}{12}s(x)\omega(y,z)+\omega(x,y)\omega(x,z)\right)Z_{M}^{(g)}.
\end{align*}
Comparing these results yields~\eqref{NablaOmega}. 
\medskip

We derive  \eqref{NablaNuDifferentialEqn} and \eqref{eq:Rauch} by integrating \eqref{NablaOmega} over the $\beta_{a} $ homology cycles.
 We first note using \eqref{eq:gamma_a.z}  and the  moduli derivatives   \eqref{eq:delael} that
\begin{align*}
\del_{b,\ell}(\gamma_{a}y)=\frac{\rho_{a}}{(y-w_{a})^{\ell-2}}\delta_{ab},
\end{align*}
for $\ell=0,1,2$ and for $a,b\in\Ip$.
Then we find using \eqref{eq:ThetaPsi} that~\cite{T1}
\begin{align}
\delCg(x)(\gamma_{a}y)&
=\left(\Psi_{2}(x,\gamma_{a}y)-\Psi_{2}(x, y)\right)d(\gamma_{a}y).
\label{NablaPolynomial}
\end{align}
Let $\nu_{b}(y)=n_{b}(y)dy$ so that $n_{b}(y)=n_{b}(\gamma_{a}y) (\gamma_{a}y)'$. 
\eqref{NablaPolynomial}  implies
\begin{align}
\delCg(x)(\gamma_{a}y)n_{b}(\gamma_{a}y)
=\left(\Psi_{2}(x,\gamma_{a}y)-\Psi_{2}(x, y)\right)\nu_{b}(y).
\label{NablaNug}
\end{align} 
Similarly letting $\omega(y,z)=f(y,z)dz$ we find that
\begin{align}\label{eq:Nablaom}
\delCg(x)(\gamma_{a}z)f(y,\gamma_{a}z)
=\left(\Psi_{2}(x,\gamma_{a}z)-\Psi_{2}(x, z)\right) \omega(y,z).
\end{align}
\eqref{PsiDifferenceBurnside} implies $\nu_{a}(y)=\oint_{\beta_{a}}\omega(y,\cdot)$ on identifying $\beta_{a}$ as a curve  from a base point $z $ to $\gamma_{a} z$. We therefore find that with $\delCgone(x)=\delCg(x)+dy\,\Psi_{2}(x,y)\partial_{y}$ then \eqref{eq:Nablaom} implies
\begin{align*}
\delCgone(x)\nu_{a}(y)&=\delCgone(x)\left(\oint_{\beta_{a}}\omega(y,\cdot )\right)
=\oint_{\beta_{a}}\delCgone(x)\omega(y, \cdot)+\delCg(x)(\gamma_{a}z)f(y,\gamma_{a}z)
\\
&=
\oint_{\beta_{a}}\delCgone(x)\omega(y,\cdot)
+\left(\Psi_{2}(x,\gamma_{a}z)-\Psi_{2}(x, z)\right)\omega(y,z)
\\
&=\oint_{\beta_{a}}\delCgone(x)\omega(y, \cdot)+ \oint_{\beta_{a}}d \left(\Psi_{2}(x,\cdot)\omega(y,\cdot)\right)
\\
&=
\omega(x,y)\nu_{a}(x)-\Psi_{2}^{(0,1)}(x,y)dy\;\nu_{a}(y),
\end{align*}
from  \eqref{NablaOmega}. Therefore \eqref{NablaNuDifferentialEqn} holds.

In a similar way, we  integrate $y$ in \eqref{NablaNuDifferentialEqn} over a $\beta_{b}$ cycle to find
\begin{align*}
\delCg(x)\left(\oint_{\beta_{b}}\nu_{a}\right)-\delCg(x) (\gamma_{b}y)n_{a}(\gamma_{b}y)+\oint_{\beta_{b}}d(\Psi_{2}(x,\cdot)\nu_{a}(\cdot))=\nu_{a}(x)\nu_{b}(x).
\end{align*}
On evaluation of the integrals and using \eqref{eq:Omdef}  and \eqref{NablaNug} we find  \eqref{eq:Rauch}.
\end{proof}
\subsection{The genus $g$ partition function for a lattice VOA }
We now find the genus $g$  partition function for the lattice VOA $V_{L}$ associated with a rank $d$ even lattice  $L$. Decomposing $V_{L}$ into  $M^{d}$ modules  
$M_{\lambda}=M^d\otimes e^{\lambda}$ we find
$V_{L}=\oplus_{\lambda\in L}M_{\lambda}$ e.g.~\cite{K}. Following \eqref{eq:Z_Walpha} we define a genus $g$ module partition function  
\begin{align*}
Z_{M_{\bm{\lambda}}}^{(g)}:=\sum_{\bm{b}_{+}\in M_{\bm{\lambda}}}Z^{(0)}(\bm{b,w}),
\end{align*}
where $M_{\bm{\lambda}}=\bigotimes_{a\in\Ip}M_{\lambda_{a}}$
and $\bm{\lambda}\in L^g$ is a $g$-tuple of lattice vectors. We find the following natural  generalization of  genus two results in~\cite{MT2, MT3}:
\begin{proposition}\label{prop:LatticeHeis1}
Let $L$ be an even lattice. Then 
\begin{align*}
Z_{M_{\bm{\lambda}}}^{(g)}=e^{\im\pi\,\bm{\lambda}.\Omega.\bm{\lambda}}\left(Z_{M}^{(g)}\right)^{d},
\end{align*}
where $Z_{M}^{(g)}$ is the rank one Heisenberg partition function and  
\begin{align*}
\bm{\lambda}.\Omega.\bm{\lambda}=\sum_{i=1}^{d}\sum_{a,b\in\Ip}\lambda_{a}^{i}\Omega_{ab}\lambda_{b}^{i},
\end{align*}
for the genus $g$ period matrix $\Omega$.
\end{proposition}
\begin{proof}
Let $h^{1}, h^{2},\ldots, h^{d}\in V_{1}$ be generators of the rank $d$ Heisenberg VOA with
\begin{align*}
[h^{i}(m),h^{j}(n)]=\delta_{ij}\delta_{m,-n}.
\end{align*}
Using Zhu recursion of Theorem~\ref{theor:ZhuGenusgModules} we find
\begin{align*}
\F_{M_{\bm{\lambda}}}^{(g)}(h^{i},x)&=\nu_{\bm{\lambda}}^{i}(x)Z_{M_{\bm{\lambda}}}^{(g)},
\\
\F_{M_{\bm{\lambda}}}^{(g)}(h^{i},x;h^{j},y)&=
\left(\nu_{\bm{\lambda}}^{i}(x) \nu_{\bm{\lambda}}^{j}(y) +\delta_{ij}\omega(x,y)\right)
Z_{M_{\bm{\lambda}}}^{(g)},
\end{align*}
for 1-form $\nu_{\bm{\lambda}}^{i}(x)=\sum_{a=1}^{g}\lambda_{a}^{i}\nu_{a}(x)$.
Taking an appropriate limit of the $2$-point function (cf. the derivation of~\eqref{eq:VirasoroProjConn}) we find that for $\omega=\frac{1}{2}\sum_{i=1}^{d}h^{i}(-1)h^{i}$ then 
\begin{align*}
Z_{M_{\bm{\lambda}}}^{(g)}(\omega,x)=\left(\frac{1}{2}\nu_{\bm{\lambda}}^{2}(x)+\frac{d}{12}s(x)\right)Z_{M_{\bm{\lambda}}}^{(g)}.
\end{align*} 
By Remark~\ref{rem:WardModule} and Proposition~\ref{prop:nablaMgZ} we also have
$
Z_{M_{\bm{\lambda}}}^{(g)}(\omega,x)=\delMg(x)Z_{M_{\bm{\lambda}}}^{(g)}$
so that 
\begin{align*}
\delMg(x)Z_{M_{\bm{\lambda}}}^{(g)} = \left(\frac{1}{2}\nu_{\bm{\lambda}}^{2}(x)+\frac{d}{12}s(x)\right)Z_{M_{\bm{\lambda}}}^{(g)}.
\end{align*}
Define $F(x,\bm{\lambda})=Z_{M_{\bm{\lambda}}}^{(g)}\left(Z_{M}^{(g)}\right)^{-d}$ with $F(x,\bm{0})=1$. Hence $\delMg(x)\log F = \frac{1}{2}\nu_{\bm{\lambda}}^{2}(x)$ 
using \eqref{prop:BosonDE}.
But  Rauch's formula \eqref{eq:Rauch} implies $\log F=\im\pi\,\bm{\lambda}.\Omega.\bm{\lambda}$   and  the result follows.
\end{proof}
The genus $g$ Siegel lattice theta function is defined by 
\begin{align*}
\Theta_{L}(\Omega):=\sum_{\bm{\lambda}\in L^{g}}e^{\im\pi\,\bm{\lambda}.\Omega.\bm{\lambda}}.
\end{align*}
Using \eqref{eq:Z_WalphaSum} we therefore find
\begin{proposition}
The genus $g$ partition function for a  lattice VOA $V_L$ is given by
\begin{align*}
Z_{V_{L}}^{(g)}=\Theta_{L}(\Omega)\left(Z_{M}^{(g)}\right)^{d}.
\end{align*}
\end{proposition}
\begin{remark}
This result is confirmed by an alternative combinatorial approach in~\cite{T2}. It appears to contradict recent results in~\cite{Co} concerning  the genus $g$ partition function for an even self-dual lattice VOA.
\end{remark}


\begin{thebibliography}{AGMV}
\bibitem[A]{A} Ahlfors, L. 	
Some remarks on  Teichm\"uller's space of Riemann surfaces,
Ann.Math. \textbf{74} (1961) 171-191.

\bibitem[AGMV]{AGMV} Alvarez-Gaumé, L., Moore, G. and Vafa, C.
Theta functions, modular invariance, and strings,
Comm.Math.Phys. \textbf{106}  1-–40 (1986).

\bibitem[Ba]{Ba} 
Baker, H.F.
\textit{Abel's Theorem and the Allied Theory Including the Theory of Theta Functions}, 
Cambridge University Press (Cambridge, 1995).

\bibitem[Be1]{Be1} 
Bers, L.
Inequalities for finitely generated Kleinian groups, 
J.Anal.Math. \textbf{18} 23--41 (1967).

\bibitem[Be2]{Be2} 
Bers, L.
Automorphic forms for Schottky groups, 
Adv.Math. \textbf{16}  332–-361 (1975).

\bibitem[Bo]{Bo} 
Bobenko, A.
Introduction to compact Riemann surfaces,
in  \textit{Computational Approach to Riemann Surfaces}, 
edited Bobenko, A. and Klein, C.,  Springer-Verlag (Berlin, Heidelberg, 2011).

\bibitem[Bu]{Bu} 
Burnside, W.
On a class of automorphic functions, 
Proc.L.Math.Soc.  \textbf{23} 49--88 (1891).

\bibitem[Co]{Co} 
Codogni, G.
Vertex algebras and Teichmüller modular forms, 
arXiv:1901.03079.

\bibitem[DGM]{DGM}
Dolan, L. and  Goddard, P.	and Montague, P.
Conformal field theories, representations and lattice constructions,
Commun. Math. Phys. \textbf{179} 61--120 (1996).

\bibitem[EO]{EO}
Eguchi, T. and Ooguri, H.
Conformal and current algebras on a general Riemann surface,
Nucl. Phys. \textbf{B282} 308--328 (1987).

\bibitem[Fa]{Fa} 
Fay, J.D. 
\textit{Theta Functions on Riemann Surfaces},
Lecture Notes in Mathematics, 
Vol. 352. Springer-Verlag, (Berlin-New York, 1973). 

\bibitem[FHL]{FHL} 
Frenkel, I.,  Lepowsky, J.  and Huang, Y.-Z.  
\textit{On Axiomatic Approaches to Vertex Operator Algebras and Modules}, 
Mem.AMS. \textbf{104} No. 494, (1993).

\bibitem[FK]{FK}  
Farkas, H.K. and  I. Kra, I.
\textit{Riemann Surfaces},
Springer-Verlag (New York, 1980).

\bibitem[Fo]{Fo} Ford, L.R.
\textit{Automorphic Functions},
AMS-Chelsea, (Providence, 2004).

\bibitem[GT]{GT} 
Gilroy, T. and Tuite, M.P. 
Genus two Zhu theory for vertex operator algebras,
arXiv:1511.07664. Under revision.

\bibitem[K]{K} 
Kac, V.
\textit{Vertex Algebras for Beginners},
Univ.Lect.Ser. \textbf{10}, AMS, 1998.

\bibitem[Li]{Li} 
Li, H.
Symmetric invariant bilinear forms on vertex  operator algebras, 
J.Pure.Appl.Alg. \textbf{96} 279--297 (1994).

\bibitem[LL]{LL} 
Lepowsky, J.  and Li, H.
\textit{Introduction to Vertex Operator Algebras and Their Representations}, 
Progress in Mathematics Vol. 227, Birkh\"auser, (Boston, 2004).

\bibitem[McI]{McI} McIntyre, A. 
Analytic torsion and Faddeev-Popov ghosts,
SUNY PhD thesis 2002, hdl.handle.net/11209/10688.

\bibitem[McIT]{McIT}
McIntyre, A. and  Takhtajan, L.A. 
Holomorphic factorization of determinants of Laplacians on Riemann surfaces and a higher genus generalization of Kronecker's first limit formula,
GAFA, Geom.Funct.Anal. \textbf{16}  1291--1323 (2006).

\bibitem[MT1]{MT1} 
Mason, G. and Tuite, M.P.
Vertex operators and modular forms, 
\textit{A Window into Zeta and Modular Physics}, eds. K. Kirsten and F. Williams, 
Cambridge University Press, (Cambridge, 2010),
MSRI Publications \textbf{57} 183--278 (2010).

\bibitem[MT2]{MT2}
Mason, G. and Tuite, M.P.
Free bosonic vertex operator algebras on genus two Riemann surfaces I, 
Commun.Math.Phys. \textbf{300} 673-–713  (2010).

\bibitem[MT3]{MT3}
Mason, G. and Tuite, M.P.
Free bosonic vertex operator algebras on genus two Riemann surfaces II,
in  \textit{Conformal Field Theory, Automorphic Forms and Related Topics}, 
Contributions in Mathematical and Computational Sciences \textbf{8} 183--225, Springer Verlag, (Berlin, Heidelberg,  2014). 

\bibitem[Mu]{Mu} 
Mumford, D.
\textit{Tata Lectures on Theta I and II},
Birkh\"{a}user, (Boston, 1983).

\bibitem[O]{O} 
Odesskii, A.
Deformations of complex structures on Riemann surfaces and integrable structures of Whitham type
hierarchies,
arXiv:1505.07779.

\bibitem[R]{R}
Rauch, H.E. 
On the  transcendental moduli of algebraic Riemann surfaces,
Proc. Nat. Acad. Sc. \textbf{11} 42--48 (1955).

\bibitem[T1]{T1}
Tuite, M.P.
Meromorphic extensions of Green's functions on a Riemann surface,
arXiv:1912.07947.

\bibitem[T2]{T2}
Tuite, M.P.
The bosonic vertex operator algebra and its modules on a genus $g$ Riemann surface. To appear.

\bibitem[TZ]{TZ}
Tuite, M.P. and  Zuevsky, A. 
The bosonic vertex operator algebra on a genus $g$ Riemann surface,
RIMS Kokyuroko \textbf{1756} 81--93 (2011).

\bibitem[Z]{Z}
Zhu, Y.
Modular-invariance of characters of vertex operator algebras,
J.Amer.Math.Soc. \textbf{9} 237--302  (1996).
\end{thebibliography}
\end{document}